\documentclass[11pt,reqno]{amsart}
\usepackage{amsmath}
\usepackage{amssymb}
\usepackage{verbatim}
\usepackage{bbm}
\usepackage{graphicx}
\usepackage[font=footnotesize]{caption}
\usepackage{subcaption}
\usepackage{epstopdf}
\usepackage{multirow}
\usepackage{url}
\usepackage{makecell}
\usepackage{tabularx}
\usepackage{longtable}
\usepackage{graphicx}
\usepackage{tikz}
\usepackage{pgfplots}
\usepackage{xcolor}
\usepackage{pgfplots}
\usepackage{pgfplotstable}
\allowdisplaybreaks

\usepackage[toc]{appendix}
\usepackage[capitalise]{cleveref}
\usepackage{color}
\allowdisplaybreaks

\usepackage[top=0.5in,bottom=0.5in,left=0.7in,right=0.7in]{geometry}

\newtheorem{lemma}{Lemma}
\newtheorem{prop}[lemma]{Proposition}

\newtheorem{theorem}[lemma]{Theorem}

\newtheorem{example}[lemma]{Example}

\newcommand{\bc}{\mathcal{B}} 
\newcommand{\nv}{\mathbf{\nu}} 
\newcommand{\id}{\mathbf{I}_d} 
\newcommand{\au}{\tilde{u}} 
\newcommand{\gl}{\lambda_n}
\newcommand{\rgl}{\mathring{\lambda}_n}
\newcommand{\tgl}{\tilde{\lambda}_n}
\newcommand{\katgl}{k\tilde{\lambda}_n}

\newcommand{\rtgl}{\mathring{\tilde{\lambda}}_n}
\newcommand{\NN}{\mathbb{N}_0}
\newcommand{\N}{\mathbb{N}}
\newcommand{\Z}{\mathbb{Z}}
\newcommand{\R}{\mathbb{R}}
\newcommand{\wh}{\widehat}
\newcommand{\be}{\begin{equation}}
\newcommand{\ee}{\end{equation}}
\newcommand{\bchi}{\mathbbm{1}}
\renewcommand{\ge}{\geqslant}
\renewcommand{\le}{\leqslant}

\DeclareMathOperator*{\arginf}{arg\,inf}

\numberwithin{equation}{section}
\numberwithin{lemma}{section}

\begin{document}
	
	\title{Sharp Wavenumber-explicit Stability Bounds for 2D Helmholtz Equations}
	
	\author{Bin Han and Michelle Michelle}
	
	\thanks{Research supported in part by
		Natural Sciences and Engineering Research Council (NSERC) of Canada under Grant RGPIN-2019-04276, and Alberta Innovates and Alberta Advanced Education}
	
	\address{Department of Mathematical and Statistical Sciences,
		University of Alberta, Edmonton,\quad Alberta, Canada T6G 2G1.
		\quad {\tt bhan@ualberta.ca}
		\quad {\tt mmichell@ualberta.ca}
	}
	
	\makeatletter \@addtoreset{equation}{section} \makeatother
	
	\begin{abstract}
	Numerically solving	the 2D Helmholtz equation is widely known to be very difficult largely due to its highly oscillatory solution, which brings about the pollution effect. A very fine mesh size is necessary to deal with a large wavenumber leading to a severely ill-conditioned huge coefficient matrix. To understand and tackle such challenges, it is crucial to analyze how the solution of the 2D Helmholtz equation depends on (perturbed) boundary and source data for large wavenumbers. In fact, this stability analysis is critical in the error analysis and development of effective numerical schemes. Therefore, in this paper, we analyze and derive several new sharp wavenumber-explicit stability bounds for the 2D Helmholtz equation with inhomogeneous mixed boundary conditions: Dirichlet, Neumann, and impedance. We use Fourier techniques, the Rellich's identity, and a lifting strategy to establish these stability bounds. Some examples are given to show the optimality of our derived wavenumber-explicit stability bounds.
	\end{abstract}
	
	
	\maketitle
	
	\pagenumbering{arabic}
	
	\section{Introduction and Motivations}
	In this paper, we shall consider the following 2D Helmholtz equation:
	\be \label{helmholtz}
	\mathcal{L}{u}:=\Delta  u + k^2 u = -f \quad \mbox{in} \quad \Omega:=(0,1)^2
	\ee
	with the following boundary conditions
	\be \label{helmholtz:bc}
	\begin{split}
		&\bc_1u  = g_{1} \quad \mbox{on} \quad \Gamma_{1}:=(0,1) \times \{0\},
		\qquad \bc_3u = g_{3} \quad \mbox{on} \quad \Gamma_{3}:=(0,1) \times \{1\},\\
		&\bc_2u  = g_{2} \quad \mbox{on} \quad \Gamma_{2}:=\{1\} \times (0,1),
		\qquad
		\bc_4u = g_{4} \quad \mbox{on} \quad \Gamma_{4}:=\{0\} \times (0,1),
	\end{split}
	\ee
where $k>0$ is a constant wavenumber,
$f \in L^{2}(\Omega)$ is the source term,
and $g_{j} \in L^{2}(\Gamma_{j})$ for $j=1,\dots,4$ are boundary data. The different boundary conditions are prescribed by boundary operators
	$\bc_1,\ldots,\bc_4$, which belong to one of
	the three boundary operators: $\id$ (i.e., $\id u=u$) for a Dirichlet boundary condition; $\frac{\partial }{\partial \nv}$ for a Neumann boundary condition; or $\frac{\partial }{\partial \nv}-ik\id$ for the impedance boundary condition, where $\nu$ is the outward normal vector. We shall assume that at least one impedance boundary condition is present. Without loss of generality, we assume that the impedance boundary condition is always imposed on $\Gamma_4$, i.e., $\bc_4=\frac{\partial }{\partial \nv}-i k \id$.
	More specifically, we are interested in the following boundary configurations
	\be \label{bc}
	\bc_1, \bc_3 \in \{\id, \tfrac{\partial }{\partial \nv}\},
	\quad
	\bc_2  \in \{\id, \tfrac{\partial }{\partial \nv}, \tfrac{\partial }{\partial \nv}-ik \id \},
	\quad \mbox{and}\quad
	\bc_4=\tfrac{\partial }{\partial \nv}-ik \id.
	\ee
	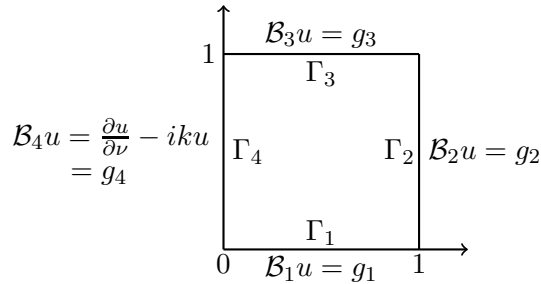
\begin{figure}[htbp]
		 \renewcommand{\figurename}{Figure}
		 \begin{tikzpicture}[yscale=0.65,xscale=0.65]
			\draw[thick] (0,0)--(4,0);
			\draw[thick] (4,0)--(4,4);
			\draw[thick] (4,4)--(0,4);
			\draw[thick] (0,4)--(0,0);
			\draw[thick,->] (0,4)--(0,5);
			\draw[thick,->] (4,0)--(5,0);
			\node (A) at (2,0.35) {$\Gamma_{1}$};
			\node (B) at (3.6,2) {$\Gamma_{2}$};
			\node (C) at (2,3.6) {$\Gamma_{3}$};
			\node (D) at (0.5,2) {$\Gamma_{4}$};
			\node (E) at (0,-0.3) {$0$};
			\node (F) at (2,-0.4) {$\mathcal{B}_{1}u=g_1$};
			\node (G) at (4,-0.3) {$1$};
			\node (H) at (2,4.4) {$\mathcal{B}_{3}u=g_3$};
			\node (I) at (5.35,2) {$\mathcal{B}_{2}u=g_2$};
			\node[align=left] (J) at (-2.3,2) {$\mathcal{B}_{4}u=\tfrac{\partial u}{\partial \nu}-iku$\\\qquad $=g_4$};
			\node (K) at (-0.3,4) {$1$};
		\end{tikzpicture}
		\caption{Boundary configuration in \eqref{helmholtz:bc} and \eqref{bc} for the 2D Helmholtz equation \eqref{helmholtz}.}\label{fig:config}
	\end{figure}

	See \cref{fig:config} for the domain and boundary configurations of the 2D Helmholtz equation  \eqref{helmholtz}--\eqref{bc}. The importance of understanding how the solutions of the Helmholtz equations behave with respect to the boundary data and the source term cannot be overstated. To the best of our knowledge, such configurations have not been studied in details and thus their stability bounds are not available except for one case, where three sides have the homogeneous Dirichlet conditions while one side has an inhomogeneous impedance boundary condition  \cite{DLS15}. Hence, our aim is to study the stability of the above 2D Helmholtz equation with inhomogeneous boundary conditions and derive several sharp wavenumber-explicit stability bounds that hold for all positive wavenumber $k$. By sharp, we mean that our stability bounds capture the leading $k$-dependent term in front of the norm of a given datum, which is accurate up to a constant multiple (independent of $k$ and the given datum). To achieve this goal, we shall also devise and rigorously justify a lifting strategy, which assists us in dealing with inhomogeneous horizontal boundary conditions. Furthermore, we shall give several examples to illustrate the sharpness of our stability bounds.
	
	Some of the above boundary configurations can be thought of as a simplification of electromagnetic scattering from a large cavity (e.g, see \cite{ABW02,BS05,BY16,BYZ12,LMS13}) or waveguide (e.g., see \cite{BG99, HK19, MMP12}) problems, where we approximate the nonlocal inhomogeneous boundary condition with an impedance boundary condition. In \cite{LMS13}, three sides have the homogeneous Dirichlet boundary conditions and one side has a nonlocal inhomogeneous boundary condition. Meanwhile, the authors in \cite{BY16} considered two different 2D Helmholtz models, where all three  sides have either the homogeneous Dirichlet or Neumann boundary conditions, and one side has a nonlocal inhomogeneous boundary condition. The model considered in \cite{HK19} has one inhomogeneous Dirichlet boundary condition, two homogeneous Neumann boundary conditions, and one high-order complete radiation boundary condition. Approximating a nonlocal boundary condition with an impedance boundary condition may change the stability behaviour \cite[Section 6]{DLS15} and it may be interesting to observe how it changes.
	
	Stability estimates for the 2D Helmholtz equations are fundamental in the error analysis and the development of various numerical schemes. There is a considerable amount of work in deriving such estimates. Many of these studies considered the interior impedance Helmholtz equation \cite{CF06,EM12,FS94,M95,S14} (i.e., only the impedance boundary condition is present in the problem). In general, obtaining a sharp wavenumber-explicit stability estimate is very challenging, since they are highly dependent on the domain and boundary configurations (e.g, \cite{EM12}). For example, we may compare the studies done by \cite{EM12,M95,S14}, which deal with a bounded Lipschitz domain. With an extra assumption that the domain is star-shaped with respect to a ball, \cite{M95} proved a stability estimate that is independent of the wavenumber $k$. In the absence of this star-shaped assumption (i.e., a general Lipschitz domain is considered), the stability estimate in \cite[Theorem 1.6]{S14} involves an extra factor $k^{\frac{1}{2}}$ in front of the boundary datum's norm and an extra factor $k$ in front of the source term's norm, and improves the stability estimate stated in \cite[Theorem 2.4]{EM12}. One of the most notable stability estimates for the interior heterogeneous Helmholtz problem with mixed boundary conditions can be found in \cite{GS20} (also see references therein for stability estimates of heterogeneous Helmholtz equations). They proved that the stability estimates may depend on the wavenumber (among other factors like the geometry of the domain and the coefficient bounds). They did not provide any explicit form of the dependence on the wavenumber $k$ except for the interior impedance problem. Another highly relevant study to our present work is \cite{H07}; in this paper, the author proved some stability estimates for the interior 2D Helmholtz equation on a bounded connected Lipschitz domain with mixed boundary conditions. Certain geometric assumptions were introduced yielding stability estimates that are independent of $k$. These same geometric assumptions were also used to prove a stability estimate for the heterogeneous Helmholtz equation in a subsequent work \cite{BGP17}. In the context of a rectangular domain, the geometric assumptions in \cite{H07} simplify into three cases: (1) all sides have impedance boundary conditions, (2) three sides have impedance boundary conditions and one side has a Dirichlet/Neumann boundary condition, or (3) two adjacent sides have Dirichlet/Neumann boundary conditions and the other two have impedance boundary conditions. By two adjacent sides, we mean two sides that are connected to each other; e.g., $\Gamma_{1}$ and $\Gamma_{2}$, or $\Gamma_{4}$ and $\Gamma_{1}$ in \eqref{helmholtz:bc}. Whenever a Dirichlet/Neumann boundary condition is imposed, \cite{H07} assumes that it is homogeneous. Hence, the boundary configurations in this paper are completely different from the assumptions used in \cite{H07}. This serves as a motivation for our present study. Moreover, as we shall see later on, our stability estimates necessarily depend on the wavenumber unlike \cite{H07}. The previous comparisons highlight the sensitivity of stability bounds with respect to boundary placements and conditions. Our results complement those in \cite{H07} for a rectangular domain, thereby offering a much more complete picture of the stability behaviour of the Helmholtz equation on a rectangular domain with mixed boundary conditions. The authors in \cite{DLS15} analyzed and numerically tested the optimal stability estimate for one particular boundary configuration we consider. While the above cited studies including our present work are concerned with interior Helmholtz problems, stability estimates for exterior Helmholtz problems (with an obstacle in the form of a bounded Lipschitz domain) can be found in for example \cite{GPS19,S14} and references therein. In the context of 1D Helmholtz equations, some stability estimates have also been derived in \cite{AKS88,GS20,IB95}, and a new numerical scheme is recently reported in \cite{HMW21} for handling arbitrarily large wavenumbers.
	
	The outline of the paper is as follows. In \cref{sec:wellposed}, we prove the existence and uniqueness of the solution to the 2D Helmholtz equation \eqref{helmholtz}--\eqref{bc}, state several sharp wavenumber-explicit stability bounds, propose a lifting strategy, and provide several examples to illustrate the optimality of our stability bounds. Finally, we present the technical proofs of several theorems in \cref{sec:proofthm}.
	
	\section{Main Results on Sharp Wavenumber-explicit Stability Bounds}
	\label{sec:wellposed}
	
	In this section, we shall study the existence and uniqueness of the solution to the Helmholtz equation \eqref{helmholtz}--\eqref{bc} and then derive several relevant sharp wavenumber-explicit stability bounds.
	
	Let $\Gamma_{D}$ be the union of all boundaries on which the Dirichlet condition is imposed (i.e., $u=g_{D}$ on $\Gamma_{D}$), $\Gamma_{N}$ be the union of all boundaries on which the Neumann condition is imposed (i.e., $\frac{\partial u}{\partial \nu} =g_{N}$ on $\Gamma_{N}$), and $\Gamma_{R}$ be the union of all boundaries on which the impedance boundary condition is imposed (i.e., $\frac{\partial u}{\partial \nu} -iku=g_{R}$ on $\Gamma_{R}$).
	
	Let $H^{s}(\Omega)$, where $s \ge 0$, be the classical Sobolev spaces of order $s$, whose norm is denoted by $\| \cdot \|_{s,\Omega}$. If $s=0$, then $H^{0}(\Omega):=L^{2}(\Omega)$, the space of square integrable complex functions over $\Omega$. The standard inner product and norm in $L^{2}(\Omega)$ are denoted by $\langle u, v \rangle_{\Omega}:= \int_{\Omega} u \overline{v}$ and $\| \cdot \|_{0,\Omega}:= \langle u, u \rangle_{\Omega}^{1/2}$.
	On the boundary, the standard inner product and norm in $L^{2}(\Gamma)$ are denoted by $\langle u, v \rangle_{\Gamma}:= \int_{\Gamma} u \overline{v}$ and $\| \cdot \|_{0,\Gamma}:= \langle u, u \rangle_{\Gamma}^{1/2}$. Define
	\[
	\mathcal{H}:=\{u \in H^{1}(\Omega)\; :\; u=0 \quad \text{on} \quad \Gamma_{D}\},
	\]
	where $\Gamma_{D}$ is allowed to be an empty set. If $\Gamma_{D}=\emptyset$, then $\mathcal{H}=H^{1}(\Omega)$. For the homogeneous Dirichlet boundary condition $u=g_D=0$ on $\Gamma_D$,
the weak formulation of the 2D Helmholtz equation \eqref{helmholtz}--\eqref{bc} is to find $u \in \mathcal{H}$ such that
	 \begin{equation}\label{weakform}
		a(u,v):=\int_{\Omega} (\nabla u \cdot \nabla \overline{v} - k^{2} u \overline{v}) - ik \int_{\Gamma_{R}}u \overline{v} =\int_{\Omega} f \overline{v} + \int_{\Gamma_{R}}g_{R}\overline{v} + \int_{\Gamma_{N}} g_{N}\overline{v} \qquad \forall\; v \in \mathcal{H}.
	\end{equation}
	The existence and uniqueness of the solution to problem \eqref{weakform} can be proved by using the Fredholm alternative and the unique continuation principle \cite[Theorem 2.1]{GS20}. For the convenience of the reader, we shall include an explicit proof for our problem to make the presentation self-contained.

	\begin{prop}
		There is a unique solution $u \in \mathcal{H}$ that satisfies the problem \eqref{weakform}.
	\end{prop}
	
	\begin{proof}
		The sesquilinear form $a(\cdot,\cdot)$ is bounded, since $|a(u,u)| \le \max(1,k^2) \|u\|_{1,\Omega}^{2}$. Also, the G\r{a}rding's inequality \cite[(2.7)]{M00} is satisfied, since $\Re(a(u,u)) = \|u\|^{2}_{1,\Omega} - (k^{2}+1) \|u\|^{2}_{0,\Omega}$. 
		We also know that $\mathcal{H}$ is compactly embedded in $L^{2}(\Omega)$. Hence, by the Fredholm alternative \cite[Theorems 2.34 and 2.27]{M00}, the solution to the variational problem \eqref{weakform} exists as long as we can show its uniqueness.
		
We now prove the uniqueness. Suppose that $f=g_{R}=g_{N}=0$ in \eqref{weakform}. We have to prove that the solution $u$ must be $0$.
Then, recalling that $\Gamma_{R}\neq \emptyset$, we have $\Im(a(u,u))=k\|u\|_{0,\Gamma_{R}}$. Since $k$ is positive, we have $u=0$ on $\Gamma_R$ almost everywhere. Let $\tilde{\Omega}$ be an extended domain $\Omega$ such that $\tilde{\Omega}=(-\varepsilon,1)\times(0,1)$ if $\Gamma_{R}=\Gamma_{4}$, $\tilde{\Omega}=(0,1+\varepsilon)\times(0,1)$ if $\Gamma_{R}=\Gamma_{2}$, or $\tilde{\Omega}=(-\varepsilon,1+\varepsilon)\times(0,1)$ if $\Gamma_{R}=\Gamma_{2} \cup \Gamma_{4}$ for some $\varepsilon>0$. Let $\au$ be the function $u$ with zero extension in $\tilde{\Omega}$. Note that $\au \in H^{1}(\tilde{\Omega})$. Since $\langle \nabla u, \nabla v\rangle_{\Omega} - k^{2}\langle u, v\rangle_{\Omega}=0$ for all $v \in  \mathcal{H}$, we have $\langle \nabla \au, \nabla v\rangle_{\tilde{\Omega}} - k^{2}\langle \au, v\rangle_{\tilde{\Omega}}=0$ for all $v \in  H^{1}(\tilde{\Omega})$. Also noting that $\au=0$ in $\tilde{\Omega} \setminus \Omega$, by \cite[Theorem 2.1]{GS20} or \cite[Theorem 1.1]{A12}, we conclude that $u=0$. The existence and uniqueness of the solution to the problem \eqref{weakform} for the Helmholtz equations have been proved.
	\end{proof}

Furthermore, the existence of a unique solution still holds true even in the presence of inhomogeneous Dirichlet boundary conditions on $\Gamma_D$ due to lifting.

\subsection{Stability bounds for inhomogeneous vertical boundary conditions}
\label{subsec:vertical}

To establish stability bounds for inhomogeneous boundary conditions only on the vertical sides, we assume the horizontal sides take homogenous boundary conditions such that
\begin{equation} \label{bc:horiz0}
	\mathcal{B}_1, \mathcal{B}_3 \in \{\id, \tfrac{\partial }{\partial \nu}\}
	\quad \mbox{and}\quad
	\mathcal{B}_{1}u =g_{1}=0\; \mbox{ on }\; \Gamma_1, \quad
	\mathcal{B}_{3}u=g_{3}=0\; \mbox{ on }\; \Gamma_3.
\end{equation}

Define $\NN:=\N\cup\{0\}$. We shall use one of the following four orthonormal bases $\{Z_{j,n}\}_{n\in \NN}$, $j=1,\ldots,4$ in $L^2(\mathcal{I})$ with $\mathcal{I}:=[0,1]$:
\be \label{Zn}
\begin{split}
&Z_{1,n}:=\sqrt{2}\sin (n\pi \cdot) \qquad \quad\;\; \mbox{and}\qquad
Z_{2,0}:=1,
\quad
Z_{2,n}:=\sqrt{2}\cos(n\pi \cdot),\qquad n\in \N,\\
&Z_{3,n}:=\sqrt{2}\sin((n+\tfrac{1}{2})\pi \cdot) \quad \mbox{and}\qquad
Z_{4,n}:=\sqrt{2}\cos((n+\tfrac{1}{2})\pi \cdot),\qquad n\in \NN.
\end{split}
\ee
To maintain a unified presentation, we often use $Z_{1,0}:=0$ instead of dropping $Z_{1,0}$. For $g\in L^2(\mathcal{I})$, we let $\tilde{g}$ be the function $g$ with the zero extension outside the interval $\mathcal{I}$, and define $2$-periodic functions $G_1,\ldots,G_4$ whose values on $(-1,1]$ are defined by
\be \label{g:G}
\begin{split}
&G_1(x):=\tilde{g}(x)-\tilde{g}(-x),
\qquad\qquad\;\;
G_2(x):=\tilde{g}(x)+\tilde{g}(-x),\quad\\
&G_3(x):=(\tilde{g}(x)-\tilde{g}(-x)) e^{ix\pi/2},\quad\;\,
G_4(x):=(\tilde{g}(x)+\tilde{g}(-x))e^{ix\pi/2}.
\end{split}
\ee
For $g\in L^2(\mathcal{I})$ and $j=1,\ldots,4$, we have
\be \label{g:fourier}
g=\sum_{n\in \NN} \wh{g}(n) Z_{j,n}
\quad \mbox{with}\quad
\wh{g}(n):=\int_{0}^1 g(x) Z_{j,n}(x) dx, \quad \forall n\in \NN,
\ee
where we used the convention $Z_{1,0}:=0$. Let 
$Z_{j,n}'$ stand for the derivative of $Z_{j,n}$.
It is also easy to observe that $\{Z_{j,n}'\}_{n\in \NN}$ is an orthogonal system in $L^2(\mathcal{I})$ satisfying $\int_0^1 Z_{j,m}'(x) Z_{j,n}'(x) dx=0$ as long as $m\ne n$. For $\{Z_{j,n}\}_{n \in \NN}$ with $j=1,2$, we let $\{\sigma_{n}=n\pi\}_{n\in \NN}$. For $\{Z_{j,n}\}_{n \in \NN}$ with $j=3,4$, we let $\{\sigma_{n}=(n+\frac{1}{2})\pi\}_{n\in \NN}$. We refer to $\{\sigma_{n}\}_{n \in \N_0}$ as eigenvalues, $\{Z_{j,n}\}_{n \in \N_0}$ as eigenfunctions, and $\{(\sigma^2_{n},Z_{j,n})\}_{n \in \N_0}$ as eigenpairs. Due to the identities in \eqref{g:fourier}, given eigenpairs $\{(\sigma^2_{n},Z_{j,n})\}_{n \in \N_0}$ for some $j \in \{1,2,3,4\}$ and $s \ge 0$, we say that $g \in \mathcal{Z}^{s}(\Gamma_\ell)$, where $\ell \in \{1,2,3,4\}$, if $g \in L^{2}(\Gamma_\ell)$ and
\begin{equation}\label{spaceYs}
	 \|g\|^2_{\mathcal{Z}^{s}(\Gamma_\ell)}:=\sum_{n=0}^{\infty}|\wh{g}(n)|^{2} \sigma_{n}^{2s} < \infty \quad \mbox{with}\quad \wh{g}(n):=\int_{\Gamma_\ell} g(x)Z_{j,n}(x) dx=\int_{0}^{1} g(x)Z_{j,n}(x) dx, \quad \forall n\in \NN.
\end{equation}
Such a Hilbert space has been commonly used in stability estimates of the Helmholtz equation; e.g, see \cite[Section 2.2]{BG99} and \cite[Section 3]{MMP12}.

In this subsection, let our eigenvalues be $\{\mu_{n}=n\pi\}_{n\in \NN}$ or $\{\mu_{n}=(n+\frac{1}{2})\pi\}_{n\in \NN}$, and our eigenfunctions be
\begin{equation} \label{1dY:sol}
	Y_{n}(y) =
	\begin{cases}
		Z_{1,n}(y), & \text{if} \quad \mu_{n}=n\pi \ne 0 \text{ and } \mathcal{B}_1=\mathcal{B}_3=\id,\\
		Z_{2,n}(y), & \text{if} \quad \mu_{n}=n\pi\ne 0 \text{ and } \mathcal{B}_1=\mathcal{B}_3=\frac{\partial}{\partial \nu},\\
		Z_{3,n}(y), & \text{if} \quad \mu_{n}=(n+\frac{1}{2})\pi \text{ and } \mathcal{B}_{1}=\id, \mathcal{B}_{3}=\frac{\partial}{\partial \nu},\\
		Z_{4,n}(y), & \text{if} \quad \mu_{n}= (n+\frac{1}{2})\pi \text{ and } \mathcal{B}_{1}=\frac{\partial}{\partial \nu}, \mathcal{B}_{3}=\id,
	\end{cases}
	\quad \forall n \in \mathbb{N}_0,
\end{equation}
which together give us $\{(\mu^{2}_n,Y_n)\}_{n\in \N_0}$ as our eigenpairs. We are now ready to state our first set of stability bounds. The oscillating part of the solution predominantly contributes to the stability bound. Hence, the main idea of the proof is to find a delicate upper bound for its norm. We achieve this by establishing several technical norm estimates. The proof of the following theorem is deferred to \cref{sec:proofthm}.
	
	 \begin{theorem}\label{thm:stability1}
		Consider the Helmholtz equation in \eqref{helmholtz}--\eqref{helmholtz:bc}. Assume that \eqref{bc:horiz0} holds, $\mathcal{B}_2 \in \{ \id, \frac{\partial }{\partial \nu}, \frac{\partial}{\partial \nu}-i k \id\}$ with $\mathcal{B}_{2} u =g_2=0$ on $\Gamma_2$, and $\mathcal{B}_4=\frac{\partial}{\partial \nu}-i k \id$ with $\mathcal{B}_4 u=g_{4} \in L^{2}(\Gamma_4)$ on $\Gamma_4$. Then the unique solution $u$ to the Helmholtz equation in \eqref{helmholtz}--\eqref{helmholtz:bc} with the source term $f$ vanishing satisfies
		\begin{equation} \label{estim:stability1}
			\|\nabla u \|_{0,\Omega} + k \|u\|_{0,\Omega} \le \sqrt{12} \max\{k,1\} \|g_{4}\|_{0,\Gamma_{4}},\qquad \forall\; k>0.
		\end{equation}
	\end{theorem}
	
	The following example demonstrates how the stability bound in \cref{thm:stability1} is sharp in the sense that the right-hand side of \eqref{estim:stability1} holds up to a constant multiple (independent of $k$ and $g_4$).
	
	\begin{example}\label{ex:stability1}
		\normalfont
		In what follows, suppose that the conditions of \cref{thm:stability1} hold, that is, $f=0$, $\mathcal{B}_1, \mathcal{B}_3 \in \{\id,\frac{\partial}{\partial \nu}\}$, and $\mathcal{B}_4=\frac{\partial}{\partial \nu}-ik\id$ by \eqref{bc} with the boundary data $g_1 = g_2= g_3=0$ in \eqref{helmholtz:bc}. Note that the boundary data $g_4=\mathcal{B}_4 u$ on $\Gamma_4$.
		We have three choices for the boundary operator $\mathcal{B}_2$. Let us consider the first case $\mathcal{B}_{2}=\frac{\partial }{\partial \nu} - ik\id$ with
		$k^{2} := \mu_{n}^2+\pi^2$, where $\mu_n=n \pi$ or $\mu_n = (n+\frac{1}{2})\pi$ for a temporarily fixed integer $n \in \mathbb{N}$.
		Let $u=\tfrac{1}{2k\pi}\left(-\sin(\pi x)k + \cos(\pi x) \pi i\right)Y_{n}(y)$ be the exact solution, where $Y_{n}(y)$ takes one of the forms in \eqref{1dY:sol}. Because $\mathcal{B}_4 =\frac{\partial}{\partial \nu}-ik\id$, we have
		$g_4=\mathcal{B}_4 u=Y_n$ on $\Gamma_4$ and hence $\|g_{4}\|_{0,\Gamma_{4}}=1$. The exact solution $u=\tfrac{1}{2k\pi}\left(-\sin(\pi x)k + \cos(\pi x) \pi i\right)Y_{n}(y)$ satisfies
		\[
		\|\nabla u \|_{0,\Omega} + k \|u\|_{0,\Omega} =  \frac{\sqrt{2}}{2}k\sqrt{\frac{1}{\pi^2}+\frac{1}{k^2}} \ge \frac{\sqrt{2}}{2\pi}k \|g_{4}\|_{0,\Gamma_{4}}
		\quad \mbox{with}\quad k:= \sqrt{\mu_n^2+\pi^2}, \quad
		\mu_n\in \{n\pi, (n+\tfrac{1}{2})\pi\}
		\]
		for all $n \in \N$.
		
		Next, we consider the second case $\mathcal{B}_{2}=\frac{\partial}{\partial \nu}$ with $k^{2} := \mu_n^2+\frac{1}{4}\pi^2$, where $\mu_n=n \pi$ or $\mu_n = (n+\frac{1}{2})\pi$ for a temporarily fixed integer $n \in \mathbb{N}$. Let $u=-\tfrac{2}{\pi}\sin(\frac{\pi}{2}x)Y_{n}(y)$ be the exact solution, where $Y_{n}(y)$ takes one of the forms in \eqref{1dY:sol}. Because $\mathcal{B}_4 =\frac{\partial}{\partial \nu}-ik\id$, $g_4=\mathcal{B}_4 u=
		Y_n$ on $\Gamma_4$ and hence $\|g_{4}\|_{0,\Gamma_{4}}=1$.
		The exact solution $u=-\tfrac{2}{\pi}\sin(\frac{\pi}{2}x)Y_{n}(y)$ satisfies
		\[
		\|\nabla u \|_{0,\Omega} + k \|u\|_{0,\Omega} =\frac{2\sqrt{2}}{\pi} k \|g_{4}\|_{0,\Gamma_{4}} \quad \mbox{with}\quad k:=\sqrt{\mu_n^2+\tfrac{1}{4}\pi^2},\quad
		\mu_n\in \{n\pi, (n+\tfrac{1}{2})\pi\},\quad \forall\; n\in \N.
		\]
		
		Finally, we consider the third case $\mathcal{B}_{2}=\id$ with $k^{2} := \mu_n^2+\pi^2$, where $\mu_n=n \pi$ or $\mu_n = (n+\frac{1}{2})\pi$ for a temporarily fixed integer $n \in \mathbb{N}$. Let $u=-\tfrac{1}{\pi}\sin(\pi x) Y_{n}(y)$ be the exact solution, where $Y_{n}(y)$ takes one of the forms in \eqref{1dY:sol}. Because $\mathcal{B}_4 =\frac{\partial}{\partial \nu}-ik\id$, $g_4=\mathcal{B}_4 u=Y_n$ on $\Gamma_4$ and hence $\|g_{4}\|_{0,\Gamma_{4}}=1$.
		The exact solution $u=-\tfrac{1}{\pi}\sin(\pi x)Y_{n}(y)$ satisfies
		\[
		\|\nabla u \|_{0,\Omega} + k \|u\|_{0,\Omega} = \frac{\sqrt{2}}{\pi} k \|g_{4}\|_{0,\Gamma_{4}}
		\quad \mbox{with}\quad k:=\sqrt{\mu_n^2+\pi^2},\quad
		\mu_n\in \{n\pi, (n+\tfrac{1}{2})\pi\},\quad \forall\; n\in \N.
		\]
		A similar example demonstrating the sharpness of the stability bound \eqref{estim:stability1} for the third case was also presented in \cite{DLS15}.
	\end{example}
	
	Next, we present our second set of stability bounds, whose proof is deferred to \cref{sec:proofthm}, since it again involves several technical norm estimates. We recall that $\mathcal{Z}^{\frac{1}{2}}(\Gamma_{2})$ can be identified with a subspace of $H^{\frac{1}{2}}(\Gamma_{2})$, which after the zero extension belongs to the space $H^{\frac{1}{2}}(\partial \Omega)$; e.g, see \cite[Section 3]{MMP12}.
	
	 \begin{theorem}\label{thm:stability2}
		Consider the Helmholtz equation in \eqref{helmholtz}--\eqref{helmholtz:bc} with the source term $f$ vanishing. Assume that \eqref{bc:horiz0} holds, $\mathcal{B}_4=\frac{\partial}{\partial \nu}-ik\id$ with $\mathcal{B}_{4} u =g_4= 0$ on $\Gamma_4$, and $\mathcal{B}_2\in \{\id, \frac{\partial}{\partial \nu},
		\frac{\partial}{\partial \nu}-ik\id\}$ with $\mathcal{B}_2 u=g_{2} \in L^{2}(\Gamma_2)$ on $\Gamma_2$. Then,
		\begin{enumerate}
			\item[(1)] For $\bc_2 =\frac{\partial}{\partial \nv}-ik\id$, the unique solution $u$ to the Helmholtz equation in \eqref{helmholtz}--\eqref{helmholtz:bc} satisfies
			\[
			\|\nabla u \|_{0,\Omega} + k \|u\|_{0,\Omega} \le \sqrt{12} \max\{k,1\} \|g_{2}\|_{0,\Gamma_{2}},\qquad \forall\; k>0.
			\]
			\item[(2)] For $\bc_2=\frac{\partial }{\partial \nv}$,
			the unique solution $u$ to the Helmholtz equation in \eqref{helmholtz}--\eqref{helmholtz:bc} satisfies
			 \begin{equation}\label{Ck:B4inneum}
				\|\nabla u \|_{0,\Omega} + k \|u\|_{0,\Omega} \le \sqrt{20} \max\{k^{2},1\}\|g_{2}\|_{0,\Gamma_2},\qquad \forall\; k>0.
			\end{equation}
			\item[(3)] For $\bc_2=\id$ and $g_{2} \in \mathcal{Z}^{\frac{1}{2}}(\Gamma_2)$, the unique solution $u$ to the Helmholtz equation in \eqref{helmholtz}--\eqref{helmholtz:bc} satisfies
			 \begin{equation}\label{Ck:B4indiric}
				\|\nabla u \|_{0,\Omega} + k \|u\|_{0,\Omega} \le \sqrt{14} \left( \max\{k^{2},1\} \|g_{2}\|_{0,\Gamma_2} + \max\{k^{\frac{1}{2}},1\} \|g_{2}\|_{\mathcal{Z}^{\frac{1}{2}}(\Gamma_2)}\right),\qquad \forall\; k>0.
			\end{equation}
		\end{enumerate}
	\end{theorem}
	
	An example demonstrating the sharpness of the stability bound in item (1) of \cref{thm:stability2} can be recovered from the first case discussed in \cref{ex:stability1}, where both vertical sides have the impedance boundary conditions with only the left hand side being inhomogeneous, by replacing $x$ in the solution $u$ with $1-x$. This way the nonzero vertical boundary condition is on the right-hand side (i.e., $\Gamma_2$). The following example demonstrates how the stability bounds \eqref{Ck:B4inneum} and \eqref{Ck:B4indiric} are sharp in the sense that the right-hand sides of \eqref{Ck:B4inneum} and \eqref{Ck:B4indiric} hold up to a constant multiple (independent of $k$ and $g_2$).
	
	\begin{example}
		\normalfont
		In what follows, suppose that the conditions of \cref{thm:stability2} hold,
		that is, $f=0$, $\mathcal{B}_1, \mathcal{B}_3\in \{\id, \frac{\partial}{\partial \nu}\}$, and
		$\mathcal{B}_4 =\frac{\partial}{\partial \nu}-ik\id$  by \eqref{bc} with the boundary data $g_1=0, g_3=0, g_4=0$ in \eqref{helmholtz:bc}. Note that the boundary data $g_2=\mathcal{B}_2 u$ on $\Gamma_2$. Let us consider the first case $\mathcal{B}_2 =\tfrac{\partial}{\partial \nu}$ with $k^{2} := \mu_n^2+\frac{1}{4}\pi^2$, where $\mu_n=n \pi$ or $\mu_n = (n+\frac{1}{2})\pi$ for a temporarily fixed integer $n \in \mathbb{N}$. Let $u=\tfrac{1}{\pi^2}\left(-2\pi \cos(\tfrac{\pi}{2}x) + 4k\sin(\tfrac{\pi}{2}x)i\right)Y_n(y)$ be the exact solution, where $Y_{n}(y)$ takes one of the forms in \eqref{1dY:sol}. Then $g_2=\mathcal{B}_2 u=Y_n$ on $\Gamma_2$ and hence
		 $\|g_{2}\|_{0,\Gamma_{2}}=1$. The exact solution $u=\tfrac{1}{\pi^2}\left(-2\pi \cos(\tfrac{\pi}{2}x) + 4k\sin(\tfrac{\pi}{2}x)i\right)Y_n(y)$
		satisfies
		\[
			\|\nabla u \|_{0,\Omega} + k \|u\|_{0,\Omega} = \frac{4\sqrt{2}}{\pi}k^{2}\sqrt{\frac{1}{\pi^2}+\frac{1}{4k^2}}
			\ge \frac{4\sqrt{2}}{\pi^2}k^2 \|g_{2}\|_{0,\Gamma_{2}},
		\]
		where $k:=\sqrt{\mu_n^2+\tfrac{1}{4}\pi^2}$ and $\mu_n\in \{n\pi, (n+\tfrac{1}{2})\pi\}$ for all $n \in \N$.
		
		Finally, we consider the case $\mathcal{B}_{2} = \id$ with $k^{2}:=\mu_n^{2} + \pi^2$, where $\mu_n=n\pi$ or $\mu_n=(n+\tfrac{1}{2})\pi$ for a temporarily fixed integer $n \in \mathbb{N}$.  Let $u=\tfrac{1}{\pi}\left(-\pi\cos(\pi x) +\sin(\pi x) k i \right)Y_n(y)$ be the exact solution, where $Y_{n}(y)$ takes the form of \eqref{1dY:sol}. Then $g_2=\mathcal{B}_2 u=Y_n$ on $\Gamma_2$ and hence
		$\|g_{2}\|_{0,\Gamma_{2}}=1$. The exact solution $u=\tfrac{1}{\pi}\left(-\pi\cos(\pi x) +\sin(\pi x) k i \right)Y_n(y)$
		satisfies
		{\small
		\[
		\|\nabla u \|_{0,\Omega} + k \|u\|_{0,\Omega} = \frac{\sqrt{2}}{\pi} \sqrt{k^4 + \pi^2 k^2} \ge \frac{1}{\pi} \left(k^2 + \pi k\right) \ge \frac{1}{\pi} \left(k^2 + \pi k^{\frac{1}{2}} \mu_n^{\frac{1}{2}}\right) \ge \frac{1}{\pi} \left(k^2 \|g_2\|_{0,\Gamma_2} + k^{\frac{1}{2}} \|g_2\|_{\mathcal{Z}^{\frac{1}{2}}(\Gamma_2)} \right),
		\]
		}
		where $k:=\sqrt{\mu_n^2+\pi^2}$ and $\mu_n\in \{n\pi, (n+\tfrac{1}{2})\pi\}$ for all $n \in \N$.
	\end{example}
	
	\subsection{Stability bounds for non-vanishing source terms}
	\label{subsec:source}

	We can derive a stability estimate for $f \in L^{2}(\Omega)$ by using the variational formulation \eqref{weakform} and the Rellich's identity \cite[Proposition 2.1]{CF06}. A part of this problem (i.e., $\Gamma_{D}=\Gamma_{1} \cup \Gamma_{2} \cup \Gamma_{3}$ and $\Gamma_{R}:=\Gamma_{4}$) was addressed in \cite[Appendix]{DLS15}. The proof of the following result is deferred to Section 3.

	\begin{theorem}
		\label{thm:stabilityf}
		Consider the Helmholtz equation \eqref{helmholtz}-\eqref{bc}. Assume that $g_{j}=0$ on $\Gamma_j$ for all $j=1,\dots,4$ and $f \in L^{2}(\Omega)$.
		\begin{itemize}
		\item[(1)] For $\mathcal{B}_{2}=\id$, if the unique solution $u$ to \eqref{weakform} is in $H^{2}(\Omega)$, then
		\be \label{estim:withf:nd}
		\|\nabla u\|_{0,\Omega} + k \| u\|_{0,\Omega} \le \sqrt{30} \max\{k^{2},1\} \|f\|_{0,\Omega}, \quad \forall k>0.
		\ee
		\item[(2)] 
		For $\mathcal{B}_{2} \in \{\frac{\partial}{\partial \nu},\frac{\partial}{\partial \nu}-ik\id\}$, if the unique solution $u$ to \eqref{weakform} is in $H^{2}(\Omega)$, then
		\be \label{estim:withf:rn}
		\|\nabla u\|_{0,\Omega} + k \| u\|_{0,\Omega} \le \sqrt{542} \max\{k^{2},k^{-1/2}\} \|f\|_{0,\Omega}, \quad \forall k >0.
		\ee
		\end{itemize}
	\end{theorem}
	The following example demonstrates how the stability bounds \eqref{estim:withf:nd} and \eqref{estim:withf:rn} are sharp in the sense that the right-hand sides of \eqref{estim:withf:nd} and \eqref{estim:withf:rn} hold up to a constant multiple (independent of $k$ and $f$).

	\begin{example}
		\normalfont In what follows, suppose that the conditions of \cref{thm:stabilityf} hold, that is, $f \in L^{2}(\Omega)$, $\mathcal{B}_1, \mathcal{B}_3 \in \{\id,\frac{\partial}{\partial \nu}\}$, $\mathcal{B}_4=\frac{\partial}{\partial \nu}-i k \id$ by \eqref{bc} with the boundary data $g_{j}=0$ on $\Gamma_j$ for all $j=1,\dots,4$. Let us consider the first case $\mathcal{B}_{2}=\id$ with $k^2:=\mu_n^2 + \pi^2$, where $\mu_n=n \pi$ or $\mu_n = (n+\frac{1}{2})\pi$ for a temporarily fixed integer $n \in \mathbb{N}$. Let $u = \frac{1}{\pi^3}\left(\pi \cos(\pi x)+\pi-2\sin(\pi x)ki\right)Y_n(y)$, where $Y_{n}(y)$ takes one of the forms in \eqref{1dY:sol}. Then $f=Y_n(y)$ and hence $\|f\|_{0,\Omega}=1$. The exact solution $u = \frac{1}{\pi^3}\left(\pi \cos(\pi x)+\pi-2\sin(\pi x)ki\right)Y_n(y)$ satisfies
		\[
		\|\nabla u \|_{0,\Omega} + k \|u\|_{0,\Omega} 
		= \frac{\sqrt{2}}{\pi^3} k^2 \left(\sqrt{1+\frac{3\pi^2}{4k^2}-\frac{\pi^4}{2k^4}} + \sqrt{1+\frac{3\pi^2}{4k^2}}\right) \ge \frac{2\sqrt{2}}{\pi^3} k^2 \|f\|_{0,\Omega},
		\]
		where $k:=\sqrt{\mu_n^2+\pi^2}$ and $\mu_n\in \{n\pi, (n+\tfrac{1}{2})\pi\}$ for all $n \in \N$.\\
		\indent Next, we consider the second case $\mathcal{B}_{2}=\frac{\partial}{\partial \nu}$ with $k^2:=\mu_n^2 + \frac{1}{4}\pi^2$, where $\mu_n=n \pi$ or $\mu_n = (n+\frac{1}{2})\pi$ for a temporarily fixed integer $n \in \mathbb{N}$. Let $u = \frac{1}{\pi^3}\left(4\pi - 8k \sin(\frac{\pi}{2}x) i\right)Y_n(y)$, where $Y_{n}(y)$ takes one of the forms in \eqref{1dY:sol}. Then $f=Y_n(y)$ and hence $\|f\|_{0,\Omega}=1$. The exact solution $u = \frac{1}{\pi^3}\left(4\pi - 8k \sin(\frac{\pi}{2}x) i\right)Y_n(y)$ satisfies
		\[
		\|\nabla u \|_{0,\Omega} + k \|u\|_{0,\Omega} 
		= \frac{4\sqrt{2}}{\pi^3} k^2 \left(\sqrt{1+\frac{\pi^2}{2k^2}-\frac{\pi^4}{8k^4}} + \sqrt{1+\frac{\pi^2}{2k^2}}\right) \ge \frac{8\sqrt{2}}{\pi^3} k^2 \|f\|_{0,\Omega},
		\]
		where $k:=\sqrt{\mu_n^2+\frac{1}{4}\pi^2}$ and $\mu_n\in \{n\pi, (n+\tfrac{1}{2})\pi\}$ for all $n \in \N$.\\
		\indent Finally, we consider the third case $\mathcal{B}_{2}=\frac{\partial}{\partial \nu}-ik\id$ with $k^2:=\mu_n^2 + \pi^2$, where $\mu_n=n \pi$ or $\mu_n = (n+\frac{1}{2})\pi$ for a temporarily fixed integer $n \in \mathbb{N}$. Let $u = \frac{1}{\pi^3}\left(\pi - k \sin(\pi x) i\right)Y_n(y)$, where $Y_{n}(y)$ takes one of the forms in \eqref{1dY:sol}. Then $f=Y_n(y)$ and hence $\|f\|_{0,\Omega}=1$. The exact solution $u = \frac{1}{\pi^3}\left(\pi - k \sin(\pi x) i\right)Y_n(y)$ satisfies
		\[
		\|\nabla u \|_{0,\Omega} + k \|u\|_{0,\Omega} 
		= \frac{1}{\sqrt{2}\pi^3} k^2 \left(\sqrt{1+\frac{2\pi^2}{k^2}-\frac{2\pi^4}{k^4}} + \sqrt{1+\frac{2\pi^2}{k^2}}\right) \ge \frac{\sqrt{2}}{\pi^3} k^2 \|f\|_{0,\Omega},
		\]
		where $k:=\sqrt{\mu_n^2+\pi^2}$ and $\mu_n\in \{n\pi, (n+\tfrac{1}{2})\pi\}$ for all $n \in \N$.
	\end{example}
	
\subsection{Stability bounds for inhomogeneous horizontal boundary conditions using a lifting technique}
\label{subsec:inhomohorizbdry}
	
	In this section, we discuss how under certain assumptions, we can transfer the inhomogeneous horizontal boundary data to the vertical boundary conditions. This procedure is well known as lifting in the literature. As we shall soon see, we are actually considering a particular instance of lifting, where our auxiliary functions do not affect the source term at all. Consider the Helmholtz equation \eqref{helmholtz}--\eqref{bc}. Without loss of generality, let us assume that only one of the horizontal boundary conditions is inhomogeneous and it is on $\Gamma_1$. We can use the same method to handle the case where both horizontal boundary conditions are inhomogeneous. Our goal is thus to explicitly construct an auxiliary function $\au$ satisfying
	 \begin{equation}\label{helmholtz:aux}
		\mathcal{L}\au:=\Delta \au + k^2 \au = 0 \quad \text{in} \quad \Omega:=(0,1)^2
		\quad \mbox{with} \quad
		\bc_1 \au  = g_1  \;\; \mbox{on} \;\; \Gamma_{1},
		\quad \bc_3 \au = 0 \;\; \mbox{on} \;\; \Gamma_{3}.
	\end{equation}
	%
	We shall impose some conditions on $g_1$ to ensure that the traces of $\au$ belong to the appropriate function spaces so that we can go back to the situations discussed in \cref{sec:wellposed}.

Expanding $g_1$ in terms of certain eigenfunctions is a vital step for the construction of the above auxiliary solution. For eigenvalues $\{\tilde{\mu}_n=n\pi \}_{n\in \NN}$, we can use either eigenfunctions $\{\tilde{X}_n=Z_{1,n}\}_{n\in \NN}$ or $\{\tilde{X}_n=Z_{2,n}\}_{n\in \NN}$. For eigenvalues $\{\tilde{\mu}_n=(n+\frac{1}{2})\pi \}_{n\in \NN}$, we can use either eigenfunctions $\{\tilde{X}_n=Z_{3,n}\}_{n\in \NN}$ or $\{\tilde{X}_n=Z_{4,n}\}_{n\in \NN}$. We first discuss how to properly choose $\tilde{\mu}_n, n\in \NN$. Define
\be \label{dbeta}
d_0:=\mbox{dist}(k^2, \pi^2 \Z)=\inf_{n\in \Z}|k^2-n\pi^2|
\quad \mbox{and}\quad
d_1:=\mbox{dist}(k^2, \pi^2 (\tfrac{1}{2}+\Z))=\inf_{n\in \Z}|k^2-(n+\tfrac{1}{2})\pi^2|.
\ee
Note that $d_0,d_1\in [0, \frac{1}{2}\pi^2]$ and $d_0+d_1=\frac{1}{2}\pi^2$.
For $n\in \NN$, we choose $\tilde{\mu}_n$ according to the following four cases:
\be \label{allcases}
\tilde{\mu}_n=
\begin{cases}
	(n+\frac{1}{2})\pi, &\text{if $\bc_1=\bc_3$ and $d_0\in [0,\frac{1}{8}\pi^2]$},\\
	n\pi, &\text{if $\bc_1=\bc_3$ and $d_0\not \in [0,\frac{1}{8}\pi^2]$},\\
	n\pi, &\text{if $\bc_1\ne \bc_3$ and $d_0\in [0, \frac{1}{8}\pi^2]\cup[\frac{3}{8}\pi^2,\frac{1}{2}\pi^2]$},\\
	(n+\frac{1}{2})\pi, &\text{if $\bc_1\ne \bc_3$ and $d_0\not \in [0,\frac{1}{8}\pi^2]\cup [\frac{3}{8}\pi^2,\frac{1}{2}\pi^2]$}.
\end{cases}
\ee

The next result states that the choices in \eqref{allcases} are critical in ensuring that the following auxiliary solution $\au$ satisfying \eqref{helmholtz:aux} is well defined. Furthermore, sufficient conditions under which the Dirichlet trace of an auxiliary function $\au$ belongs to $H^{\frac{1}{2}}(\partial \Omega)$ and the Neumann trace in the $x$-direction of an auxiliary function $\au$ belongs to $L^{2}(\partial \Omega)$ are presented. This allows us to fall back to the cases discussed in \cref{sec:wellposed}; more specifically, with $g_j$ replaced by $g_j - \mathcal{B}_j(\au)$ on $\Gamma_j$ for each $j \in \{2,4\}$. The proof of the following result is deferred to \cref{sec:proofthm}.
	
	\begin{prop}
		\label{traceL2}
		Assume $g_{1} \in \mathcal{Z}^{\frac{1}{2}}(\Gamma_1)$ if $\mathcal{B}_{1}=\frac{\partial}{\partial \nu}$. Otherwise, assume $g_{1} \in \mathcal{Z}^{\frac{3}{2}}(\Gamma_1)$ if $\mathcal{B}_{1}=\id$. Suppose that $\{\tilde{\mu}_n\}_{n\in\NN}$ are chosen according to \eqref{allcases}. Let the auxiliary function $\au$ take the following form
		\begin{equation} \label{auxsol}
			\au=\sum_{n=0}^{\infty} \wh{g}_{1}(n)\tilde{X}_{n}(x)\tilde{Y}_{n}(y) \quad \mbox{with}\quad \wh{g}_{1}(n):=\int_{\Gamma_1}g_{1}(x) \tilde{X}_{n}(x) dx,
		\end{equation}
		where $\{Y_n\}_{n \in \NN}$ solve
		\begin{align}
			\label{1dYtil}
			& \tilde{Y}_{n}''(y)+(k^{2}-\tilde{\mu}_{n}^{2})\tilde{Y}_{n}(y) = 0 \quad \text{in} \quad \mathcal{I}:=(0,1), \quad n\in \N_{0},\\
			\label{1dYtil:bc}
			& \mathcal{B}_{1}\tilde{Y}_{n}(0)=1, \quad \mathcal{B}_{3}\tilde{Y}_{n}(1) = 0,
		\end{align}
		with $\mathcal{B}_1,\mathcal{B}_3 \in \{\id,\tfrac{\partial}{\partial \nu}\}$.
		Then, the auxiliary function $\au$ in \eqref{auxsol} satisfies \eqref{helmholtz:aux} and each term of $\au$ is well defined. Furthermore, we have that $\au \in H^{1}(\Omega)$, the (Dirichlet) trace of $\au$ is in $H^{\frac{1}{2}}(\partial \Omega)$, and the trace of $\au_x$ (i.e., the Neumann trace in the $x$-direction of $\au$) is in $L^{2}(\partial \Omega)$.
	\end{prop}
	
	Next, we study upper bounds of an auxiliary function satisfying \eqref{helmholtz:aux} defined in \eqref{auxsol}.
	
	 \begin{theorem}\label{thm:stability3}
		Consider an auxiliary function $\au$ satisfying \eqref{helmholtz:aux} defined in \eqref{auxsol}, which takes into account of \eqref{allcases}. Then,
		\begin{itemize}
			\item[(1)]
			For $\mathcal{B}_{1}=\tfrac{\partial}{\partial \nu}$, $\mathcal{B}_{3}\in \{\id,\tfrac{\partial}{\partial \nu}\}$, and $g_{1} \in L^{2}(\Gamma_1)$, the auxiliary function $\au$ satisfies
			\begin{equation} \label{Ck:auxsol:neum}
				\|\nabla \au\|_{0,\Omega} + k \|\au\|_{0,\Omega} \le 2\sqrt{717} \max\{k,1\} \|g_{1}\|_{0,\Gamma_{1}}, \quad \forall k>0.
			\end{equation}
			\item[(2)]
			For $\mathcal{B}_{1}=\id$, $\mathcal{B}_{3}\in \{\id,\tfrac{\partial}{\partial \nu}\}$, and $g_{1} \in \mathcal{Z}^{\frac{1}{2}}(\Gamma_1)$, the auxiliary function $\au$ satisfies
			 \begin{equation}\label{Ck:auxsol:diric}
				\|\nabla \au \|_{0,\Omega} + k \| \au \|_{0,\Omega} \le 2\sqrt{43} \left( \max\{k^2,1\}\|g_{1}\|_{0,\Gamma_{1}} +\max\{k^{\frac{1}{2}},1\}\|g_{1}\|_{\mathcal{Z}^{\frac{1}{2}}(\Gamma_{1})}\right), \quad \forall k>0.
			\end{equation}
		\end{itemize}
	\end{theorem}
	
	Note that by symmetry, the same results as above hold when $\mathcal{B}_{1} \tilde{u}=0$ on $\Gamma_1$ and $\mathcal{B}_3 \au=g_3$ on $\Gamma_3$ in \eqref{helmholtz:aux}. Also, the conditions imposed on $g_1$ in \cref{thm:stability3} are weaker compared to those in \cref{traceL2}, because in the former, we are only interested in finding an upper bound of the norm of an auxiliary solution and do not consider whether its traces belong to particular spaces or not. The following example demonstrates how the stability bounds \eqref{Ck:auxsol:neum} and \eqref{Ck:auxsol:diric} are sharp in the sense that the right-hand sides of \eqref{Ck:auxsol:neum} and \eqref{Ck:auxsol:diric} hold up to a constant multiple (independent of $k$ and $g_1$).
	
	\begin{example}
		\normalfont
		In what follows, suppose that the conditions of \cref{thm:stability3} hold.
		Note that the source term in \eqref{helmholtz:aux} vanishes. Suppose that $\mathcal{B}_1=\mathcal{B}_3=\tfrac{\partial}{\partial \nu}$ and $g_1 \in L^{2}(\Gamma_1)$. Consider $k^2:=(n\pi)^2 + \tfrac{1}{4}\pi^2$ for a temporarily fixed integer $n \in \mathbb{N}$. Since $d_0 \notin [0,\tfrac{1}{8}\pi^2]$ in \eqref{allcases}, let $\au = - \tfrac{2}{\pi} \tilde{X}_n(x) \cos(\tfrac{\pi}{2}(y-1))$ be an auxiliary solution, where $\tilde{X}_n(x)=\sqrt{2}\cos(\tilde{\mu}_n x)$ or $\tilde{X}_n(x)=\sqrt{2}\sin(\tilde{\mu}_n x)$ with $\tilde{\mu}_n=n\pi$. Then $g_1 = \mathcal{B}_1 \au =\tilde{X}_n$ on $\Gamma_1$ and hence $\|g_1\|_{0,\Gamma_1}=1$. The auxiliary solution $\au = - \tfrac{2}{\pi} \tilde{X}_n(x) \cos(\tfrac{\pi}{2}(y-1))$ satisfies
		\[
		\| \nabla \au\|_{0,\Omega} + k \|\au\|_{0,\Omega} = \frac{2\sqrt{2}}{\pi} k  \|g_1\|_{0,\Gamma_1} \quad \text{with} \quad k:=\sqrt{(n\pi)^2 +\tfrac{1}{4} \pi^2}, \quad \forall n \in \mathbb{N}.
		\]
		
		Suppose that $\mathcal{B}_1=\tfrac{\partial}{\partial \nu}$, $\mathcal{B}_3=\id$, and $g_1 \in L^{2}(\Gamma_1)$. Consider $k^2:=(n\pi)^2 + \pi^2$ for a temporarily fixed integer $n \in \mathbb{N}$. Since $d_0 \in [0,\tfrac{1}{8}\pi^2]$, let $\au = \tfrac{1}{\pi}\tilde{X}_n(x) \sin(\pi(y-1))$ be an auxiliary solution, where $\tilde{X}_n(x)=\sqrt{2}\cos(\tilde{\mu}_n x)$ or $\tilde{X}_n(x)=\sqrt{2}\sin(\tilde{\mu}_n x)$ with $\tilde{\mu}_n=n\pi$. Then $g_1 = \mathcal{B}_1 \au =\tilde{X}_n$ on $\Gamma_1$ and hence $\|g_1\|_{0,\Gamma_1}=1$. The auxiliary solution $\au = \tfrac{1}{\pi}\tilde{X}_n(x) \sin(\pi(y-1))$ satisfies
		\[
		\| \nabla \au\|_{0,\Omega} + k \|\au\|_{0,\Omega} = \frac{\sqrt{2}}{\pi} k \|g_1\|_{0,\Gamma_1} \quad \text{with} \quad k:= \sqrt{(n\pi)^2+\pi^2}, \quad \forall n \in \mathbb{N}.
		\]
		
		Suppose that $\mathcal{B}_1=\id$, $\mathcal{B}_3=\tfrac{\partial}{\partial \nu}$, and $g_1 \in \mathcal{Z}^{\frac{1}{2}}(\Gamma_1)$. Consider $k^{2}:=(\theta_n + \theta_n^{-1})^{2} + (\theta_n - \tfrac{1}{2}\pi)^2$, where $\theta_n:=(n+\tfrac{1}{2})\pi$ for a temporarily fixed integer $n \in \mathbb{N}$. Since $d_0 \in [\tfrac{3}{8}\pi^2,\tfrac{1}{2}\pi^2]$ in \eqref{allcases}, let $\tilde{u} = \tilde{X}_n (x) \tfrac{\cos((\theta_n + \theta_n^{-1})(y-1))}{\cos((\theta_n + \theta_n^{-1}))}$ be an auxiliary solution, where $\tilde{X}_n(x)=\sqrt{2} \sin(\tilde{\mu}_n x)$ or $\tilde{X}_n(x)=\sqrt{2} \cos(\tilde{\mu}_n x)$ with $\tilde{\mu}_n=n\pi$. Then $g_1 = \mathcal{B}_1 \au =\tilde{X}_n$ on $\Gamma_1$ and hence $\|g_1\|_{0,\Gamma_1}=1$. The auxiliary solution $\tilde{u} = \tilde{X}_n(x) \tfrac{\cos((\theta_n + \theta_n^{-1})(y-1))}{\cos((\theta_n + \theta_n^{-1}))}$ satisfies
		\begin{align*}
		\| \nabla \au\|_{0,\Omega}^2 & + k^2 \| \au \|_{0,\Omega}^{2} = \frac{k^{2} + \tilde{\mu}_n^2 \tfrac{\sin(2(\theta_n+\theta_n^{-1}))}{2(\theta_n+\theta_n^{-1})}}{\cos^{2}(\theta_n + \theta_n^{-1})} = \frac{k^{2} + \tilde{\mu}_n^2 \tfrac{\sin(2(\theta_n+\theta_n^{-1}))}{2(\theta_n+\theta_n^{-1})}}{\sin^{2}(\theta_n^{-1})} \ge \theta_n^{2} k^{2} + \tilde{\mu}_n^2 \theta_n^2 \tfrac{\sin(2(\theta_n+\theta_n^{-1}))}{2(\theta_n+\theta_n^{-1})}\\
		& = \theta_n^{2}(\theta_n + \theta_n^{-1})^2 + \tilde{\mu}_n^2 \theta_n^2 \left(1+\frac{\sin(2(\theta_n+\theta_n^{-1}))}{2(\theta_n+\theta_n^{-1})}\right)\\
		& \ge \min\left\{\frac{\theta_n^{2}(\theta_n + \theta_n^{-1})^2}{((\theta_n + \theta_n^{-1})^{2} + \tilde{\mu}_n^2)^2}, \frac{\tilde{\mu}_n^2 \theta_n^{2}}{2((\theta_n + \theta_n^{-1})^{2} + \tilde{\mu}_n^2)} \right\} (k^4 + k^2)\\
		& \ge \min\left\{\frac{\theta_n^{2}}{4(\theta_n + \theta_n^{-1})^2}, \frac{\pi^2 \theta_n^{2}}{4(\theta_n + \theta_n^{-1})^2} \right\} (k^4 + k \tilde{\mu}_n)\\
		& = \frac{\theta_1^{2}}{4(\theta_1 + \theta_1^{-1})^2} (k^4 + k \tilde{\mu}_n) = \frac{81\pi^4}{4(9\pi^2+4)^2}(k^4 + k \tilde{\mu}_n),
		\end{align*}
		where we used the fact that $|\sin(x)| \le |x|$ for all $x \ge 0$ to arrive at the first inequality. Using the basic inequality $a^2+b^2\ge \tfrac{1}{\sqrt{2}}(a+b)$ for nonnegative numbers $a$ and $b$, we have
		\[
		\| \nabla \au\|_{0,\Omega} + k \| \au \|_{0,\Omega} \ge \frac{9\pi^2}{2\sqrt{2}(9\pi^2 + 4)} \left(k^2 \|g_1\|_{0,\Gamma_1} + k^{\frac{1}{2}} \|g_1\|_{\mathcal{Z}^{\frac{1}{2}}(\Gamma_1) }\right),
		\]
		where $k:=\sqrt{(\theta_n + \theta_n^{-1})^{2} + \tilde{\mu}_n^2}$ with $\theta_n = (n+\tfrac{1}{2})\pi$ and $\tilde{\mu}_n=\theta_n-\tfrac{1}{2}\pi$ for all $n \in \mathbb{N}$.
		
		Suppose that $\mathcal{B}_1=\id$, $\mathcal{B}_3=\id$, $g_1 \in \mathcal{Z}^{\frac{1}{2}}(\Gamma_1)$, and $g_3=0$. Consider $k^{2}:=(\theta_n + \theta_n^{-1})^{2} + \theta_n^2$, where $\theta_n:=n\pi$ for a temporarily fixed integer $n \in \mathbb{N}$. Since $d_0 \notin [0,\tfrac{1}{8} \pi^2]$ in \eqref{allcases}, let $\au = -\tilde{X}_n(x) \tfrac{\sin((\theta_n + \theta_n^{-1})(y-1))}{\sin(\theta_n + \theta_n^{-1})}$ be an auxiliary solution, where $\tilde{X}_n=\sqrt{2} \sin(\tilde{\mu}_n x)$ or $\tilde{X}_n=\sqrt{2} \cos(\tilde{\mu}_n x)$ with $\tilde{\mu}_n=\theta_n=n\pi$. Then $g_1 = \mathcal{B}_1 \au =\tilde{X}_n$ on $\Gamma_1$ and hence $\|g_1\|_{0,\Gamma_1}=1$. The auxiliary solution $\tilde{u} = -\tilde{X}_n(x) \tfrac{\sin((\theta_n + \theta_n^{-1})(y-1))}{\sin(\theta_n + \theta_n^{-1})}$ satisfies
		\begin{align*}
			\| \nabla \au \|_{0,\Omega}^2 & + k^2 \| \au \|_{0,\Omega}^2 = \frac{k^2 - \theta_n^2 \tfrac{\sin(2(\theta_n + \theta_n^{-1}))}{2(\theta_n + \theta_n^{-1})}}{\sin^{2}(\theta_n + \theta_n^{-1})}
			=\frac{(\theta_n + \theta_n^{-1})^2 + \theta_n^2 \left(1 - \tfrac{\sin(2(\theta_n + \theta_n^{-1}))}{2(\theta_n + \theta_n^{-1})}\right)}{\sin^{2}(\theta_n^{-1})}\\
			& \ge \theta_n^2 (\theta_n + \theta_n^{-1})^2 + \theta_n^4 \left(1 - \frac{\sin(2(\theta_n + \theta_n^{-1}))}{2(\theta_n + \theta_n^{-1})}\right)\\
			& \ge \frac{\theta_1^{2}}{4(\theta_1 + \theta_1^{-1})^2} (k^4 + k\tilde{\mu}_n) = \frac{\pi^4}{4(\pi^2+1)^2}(k^4 + k\tilde{\mu}_n),
		\end{align*}
		where we used the same steps as in the previous case, and the fact that $\theta_1 + \theta_1^{-1} > \pi$ and hence $1 - \tfrac{\sin(2(\theta_n + \theta_n^{-1}))}{2(\theta_n + \theta_n^{-1})} > \tfrac{1}{2}$ for all $n \ge 1$ to move from the first inequality to the second inequality. Using the basic inequality $a^2+b^2\ge \tfrac{1}{\sqrt{2}}(a+b)$ for nonnegative numbers $a$ and $b$, we have
		\[
		\| \nabla \au\|_{0,\Omega} + k \| \au \|_{0,\Omega} \ge \frac{\pi^2}{2\sqrt{2}(\pi^2 + 1)} \left(k^2 \|g_1\|_{0,\Gamma_1} + k^{\frac{1}{2}} \|g_1\|_{\mathcal{Z}^{\frac{1}{2}}(\Gamma_1) }\right),
		\]
		where $k:=\sqrt{(\theta_n + \theta_n^{-1})^{2} + \theta_n^2}$ with $\theta_n=\tilde{\mu}_n=n\pi$ for all $n \in \mathbb{N}$.
	\end{example}
	
	We close this section with an important final remark. By the superposition principle, the stability bounds for the case where all boundary conditions are inhomogeneous and the source term vanishes can be recovered by using \cref{thm:stability3}, subtracting the traces of the auxiliary solutions from $g_2$ on $\Gamma_2$ and $g_4$ on $\Gamma_4$, using \cref{thm:stability1,thm:stability2}, and finally adding all these bounds. Additionally, if the source term is nonzero, then we also add the stability bound in \cref{thm:stabilityf}.
	
	\section{Proofs of \cref{thm:stability1,thm:stability2,thm:stabilityf,thm:stability3,lem:tYn,traceL2}}
	\label{sec:proofthm}
	
	To prove \cref{thm:stability1}, we need the following result.
	
\begin{lemma}\label{1dX:sol}
Let $\{X_n\}_{n\in \NN}$ be the solutions of	 
the following problem
\begin{align}
	& \label{1dX}
	 X_{n}''(x)+(k^{2}-\mu_{n}^{2})X_{n}(x) = 0 \quad \text{in} \quad \mathcal{I}:=(0,1), \quad n\in \N_{0},\\
	& \label{1dX:bc:g4non0}
	 \mathcal{B}_{4}X_{n}(0)=\delta_{j,4}, \quad \mathcal{B}_{2}X_{n}(1)=\delta_{j,2},
\end{align}
where $\{\mu_{n}=n\pi\}_{n\in \NN}$ or $\{\mu_{n}=(n+\frac{1}{2})\pi\}_{n\in \NN}$ with $j\in \{2,4\}$, $\delta_{j,j}=1$, and $\delta_{j,m}=0$ for $j\neq m$. Define
\[
\lambda_{n}:=\sqrt{\left|1-\tfrac{\mu_{n}^{2}}{k^2}\right|} \quad \text{and} \quad
\mathring{\lambda}_{n}:=
\begin{cases}
	\lambda_{n} &\text{if }\; \mu_n^2 \le k^2,\\
	i \lambda_{n} &\text{if }\; \mu_n^2 > k^2,
\end{cases}
\quad
\forall n \in \NN.
\]
Recall that $\bc_4=\frac{\partial}{\partial \nu}-ik \id$ in \eqref{bc}.
\begin{itemize}
\item[(1)]  If $\mathcal{B}_{4}X_{n}(0)=-X'_{n}(0) -ik X_{n}(0)=1$
and $\mathcal{B}_{2}X_{n}(1)=X_{n}'(1)-ikX_{n}(1)=0$ with $\bc_2=\frac{\partial}{\partial \nu}-ik \id$,
then the solutions $\{X_n\}_{n\in \NN}$ to the problem \eqref{1dX}--\eqref{1dX:bc:g4non0} are given by
\begin{align*}
	X_{n}(x) & =	
	 \tfrac{\rgl((1-\rgl^{2})\cos(k\rgl)\sin(k\rgl(1-x))
	-(1+\rgl^2)\sin(k \rgl x))}
	 {k(4\rgl^2+(1-\rgl^{2})^2\sin^2(k\rgl))}
	  + \tfrac{(1+\rgl^{2}) \cos(k\rgl x)-(1-\rgl^{2})\cos(k\rgl) \cos(k\rgl(1-x))}
	 {k(4\rgl^2+(1-\rgl^{2})^2\sin^2(k\rgl))}
	 i.
\end{align*}
If $\mathcal{B}_{4}X_{n}(0)=-X'_{n}(0) -ik X_{n}(0)=0$ and $\mathcal{B}_{2}X_{n}(1)
=X_{n}'(1)-ikX_{n}(1)=1$ with $\bc_2=\frac{\partial }{\partial \nu}-\id$, then the solutions $\{X_n\}_{n\in \NN}$ to
\eqref{1dX}--\eqref{1dX:bc:g4non0} are given above with $x$ replaced by $1-x$.
Moreover, for both cases,
the norms $\|X_n\|_{0,\mathcal{I}}$ and
$\|X_n'\|_{0,\mathcal{I}}$  are given by
\begin{align*}
\|X_{n}\|^{2}_{0,\mathcal{I}} =
\tfrac{k\rgl(1+\rgl^2)-(1-\rgl^2)\cos(k\rgl)\sin(k\rgl)}
{2k^3 \rgl(4 \rgl^2+(1-\rgl^2)^2\sin^2(k \rgl))}, \quad
 \|X'_{n}\|^{2}_{0,\mathcal{I}} =
\tfrac
{k\rgl^2(1+\lambda_{n}^2)-\rgl(\rgl^2-1)\cos(k\rgl)\sin(k\rgl)}
{2k(4\rgl^2+(1-\rgl^2)^2\sin^2(k\rgl))}.
\end{align*}

\item[(2)] If $\mathcal{B}_{4}X_{n}(0)=1$ and $\mathcal{B}_{2}X_{n}(1)=\alpha X_{n}'(1) + (1-\alpha)X_{n}(1)=0$ with $\alpha \in \{0,1\}$, then the solutions $\{X_n\}_{n\in \NN}$ to the problem \eqref{1dX}--\eqref{1dX:bc:g4non0} are given by
\begin{align*}
	X_n(x) =
	 \tfrac{-\rgl(\cos(k\rgl)\sin(k\rgl(1-x))+\alpha\sin(k\rgl x))}
	{k((1-\rgl^2)\cos^2(k\rgl) + \rgl^2-(1-\alpha)(1+\rgl^2))}
	 +\tfrac{\cos(k\rgl)\cos(k\rgl(1-x)) -(1-\alpha) \cos(k\rgl x)}{k((1-\rgl^2)\cos^2(k\rgl) + \rgl^2-(1-\alpha)(1+\rgl^2))}i.
\end{align*}
Moreover, the norms $\|X_n\|_{0,\mathcal{I}}$ and
$\|X_n'\|_{0,\mathcal{I}}$  are given by
\begin{align*}
\|X_{n}\|^{2}_{0,\mathcal{I}} =
	\tfrac{\sin(k\rgl) \cos(k\rgl) +(-1)^{1-\alpha}k\rgl}
	{2k^3 \rgl \left(\left(1-\rgl^2 \right)
	 \cos^{2}(k\rgl)+\lambda_n^{2}-(1-\alpha)(1+\lambda_{n}^{2}) \right)},
\quad
\|X'_{n}\|^{2}_{0,\mathcal{I}} =
		\tfrac{\rgl\left( -\sin \left(k\rgl \right) \cos \left(k\rgl\right) + (-1)^{1-\alpha} k\rgl \right)}
		{2k\left(\cos^{2}\left(k \rgl \right)(1 -
		 \gl^{2})+\gl^{2}-(1-\alpha)(1+\lambda_n^2)\right)}.
\end{align*}

\item[(3)] If $\mathcal{B}_{4}X_{n}(0)=0$ and $\mathcal{B}_{2}X_{n}(1)=\alpha X_{n}'(1) + (1-\alpha)X_{n}(1)=1$ with $\alpha \in \{0,1\}$, then the solutions $\{X_n\}_{n\in \NN}$ to the problem \eqref{1dX}--\eqref{1dX:bc:g4non0} are given by
\begin{align*}
	X_n(x) & =
	\tfrac{\sin(k\rgl x)(\alpha \cos(k\rgl) -(1-\alpha) \sin(k\rgl)) - \cos(k\rgl x) \rgl^2 (\alpha \sin(k\rgl) + (1-\alpha) \cos(k\rgl))}
	{(k \rgl)^\alpha ((1-\rgl^2)\cos^2(k\rgl)+\rgl^2-(1-\alpha)(1+\rgl^2))}\\
	& \quad +\tfrac{\alpha k^{-\alpha} \cos(k\rgl(1-x))-(1-\alpha)k^{-2\alpha}\rgl^{1-\alpha}\sin(k\rgl(1-x))}
	 {(1-\rgl^2)\cos^2(k\rgl)+\rgl^2-(1-\alpha)(1+\rgl^2)}i.
\end{align*}
Moreover, the norms $\|X_n\|_{0,\mathcal{I}}$ and
$\|X_n'\|_{0,\mathcal{I}}$  are given by
\begin{align*}
	& \|X_{n}\|^{2}_{0,\mathcal{I}} =
	 \tfrac{k\rgl(1+\rgl^2)-(1-\rgl^2)\cos(k\rgl)\sin(k\rgl)}
	{2(k\rgl)^{2\alpha+1}(\alpha \rgl^2 + (1-\alpha) + (-1)^{1-\alpha} (1-\rgl^2) \cos^2(k\rgl))},\\
	& \|X'_{n}\|^{2}_{0,\mathcal{I}} =
	\tfrac{k\rgl (1 + \rgl^2) +(1 - \rgl^2) \sin(k\rgl) \cos(k\rgl)}{2(k\rgl)^{2\alpha-1} (\alpha \rgl^2 + (1-\alpha) + (-1)^{1-\alpha} (1-\rgl^2) \cos^2(k\rgl))}.
\end{align*}
\end{itemize}
\end{lemma}

\begin{proof}
	Recall from the standard ordinary differential equation theory that the solution to \eqref{1dX}-\eqref{1dX:bc:g4non0} for each $n \in \N_{0}$ with $\mu_n^2 < k^2$ takes the form
	\be \label{1dX:AB}
	X_{n}(x) = A_n \exp(ik\lambda_n x) + B_n \exp(-ik\lambda_n x),
	\ee
	where $A_{n},B_{n}$ are uniquely determined by imposing the boundary conditions.
	Then,
	\[
	\|X_n\|^2_{0,\mathcal{I}}= \int_0^1 |\Re(X_n)|^2 + |\Im(X_n)|^2 dx \quad \text{and} \quad
	\|X'_n\|^2_{0,\mathcal{I}}= \int_0^1 |\Re(X'_n)|^2 + |\Im(X'_n)|^2 dx.
	\]
	For $n \in \NN$ with $\mu_n^2 > k^2$, each solution $X_n$ and its norms can be directly obtained by replacing $\gl$ with $i\gl$ in \eqref{1dX:AB}. For $n \in \NN$ such that $\mu_n^2 = k^2$, the solution $X_n$ and its norms can be obtained by letting $\gl$ tend to zero in \eqref{1dX:AB}.
\end{proof}

	The following quantities will be used numerous times in the proofs of \cref{thm:stability1,thm:stability2}. Similar quantities will also be used multiple times in the proof of \eqref{thm:stability3}. Let $\{\mu_n = n\pi\}_{n \in \NN}$ or $\{\mu_n= (n+\frac{1}{2})\pi\}_{n \in \NN}$ and $\{X_n\}_{n\in \NN}$ be solutions to \eqref{1dX}-\eqref{1dX:bc:g4non0} with boundary conditions explicitly given in the proofs. Define $N_p:=\max\{n \in \mathbb{N}_0: \mu_n^2<k^2\}$, $N_c \in \mathbb{N}$ such that $\mu_{N_c}^2=k^2$, $N_e:=\min\{n\in \mathbb{N}_0: \mu_n^{2} > k^2\}$, and
	\begin{align} \nonumber
		\phi_n & :=\|X'_{n}\|^{2}_{0,\mathcal{I}} + (\mu_{n}^{2} + k^{2}) \|X_{n}\|^{2}_{0,\mathcal{I}}, \quad  0 \le n \le N_p,\\
		\label{phinpsin}
		\theta_{N_c} & := \|X'_{N_c}\|^{2}_{0,\mathcal{I}} + (\mu_{N_c}^{2} + k^{2})\|X_{N_c}\|^{2}_{0,\mathcal{I}},\\
		\nonumber
		\psi_n & :=\|X'_{n}\|^{2}_{0,\mathcal{I}} + (\mu_{n}^{2} + k^{2}) \|X_{n}\|^{2}_{0,\mathcal{I}}, \quad n \ge N_e.
	\end{align}
	Also recall that if $a,b\ge 0$, then the following inequality always holds
	 \begin{equation}\label{ineq:ab}
		\sqrt{a^{2}+b^{2}} \le a+b \le \sqrt{2}\sqrt{a^2 + b^2}.
	\end{equation}
	Let $\bchi_A$ denote the indicator function of the set $A$.
	
	\begin{proof}[Proof of \cref{thm:stability1}]
		Given the boundary assumptions, the solution $u$ can be expressed as
		\begin{equation*} 
			u = \sum_{n=0}^{\infty} \wh{g_{4}}(n) X_{n}(x) Y_n(y) \quad \text{with} \quad \wh{g_{4}}(n):=\int_{\Gamma_4}g_{4}(y)Y_{n}(y) dy,
		\end{equation*}
		where $\{X_{n}\}_{n \in \N_{0}}$ are stated in \cref{1dX:sol} and $\{Y_{n}\}_{n \in \N_{0}}$ are stated in \eqref{1dY:sol}.
		
		Recall that  $\gl=\sqrt{\left|1-\frac{\mu_{n}^{2}}{k^2}\right|}$ and observe that $\|Y_n'\|_{0,\mathcal{I}}=\mu_n$. By \eqref{Zn}, since both $\{Y_{n}\}_{n \in \NN}$ and $\{Y_{n}'\}_{n \in \N}$ are orthogonal systems in $L^{2}(\mathcal{I})$, we deduce that
		\begin{align}
			\nonumber
			\|\nabla u \|^{2}_{0,\Omega} & + k^{2} \|u\|^{2}_{0,\Omega}
			\nonumber
			= \left\|\sum_{n=0}^{\infty} \wh{g_{4}}(n) X'_{n} Y_{n}\right\|^{2}_{0,\Omega} + \left\|\sum_{n=0}^{\infty} \wh{g_{4}}(n) X_{n}Y'_{n}\right\|^{2}_{0,\Omega} + k^{2} \left\|\sum_{n=0}^{\infty} \wh{g_{4}}(n) X_{n}Y_{n}\right\|^{2}_{0,\Omega}\\
			& \label{ubg}
			= \sum_{n=0}^{\infty} |\wh{g_{4}}(n)|^{2} \|X'_{n}\|^{2}_{0,\mathcal{I}} + \sum_{n=0}^{\infty} |\wh{g_{4}}(n) \mu_{n}|^{2} \|X_{n}\|^{2}_{0,\mathcal{I}} + k^{2} \sum_{n=0}^{\infty} |\wh{g_{4}}(n)|^{2} \|X_{n}\|^{2}_{0,\mathcal{I}}\\
			& \nonumber
			= \sum_{n=0}^{\infty} |\wh{g_{4}}(n)|^{2} \left(\|X'_{n}\|^{2}_{0,\mathcal{I}} + (\mu_{n}^{2} +k^2) \|X_{n}\|^{2}_{0,\mathcal{I}}\right)
			\le
			\max\left\{\max_{0 \le n \le N_p} \phi_{n},
			\theta_{N_c},
			\max_{n \ge N_e} \psi_{n}\right\}\sum_{n=0}^{\infty} |\wh{g_{4}}(n)|^{2},
		\end{align}
		where $\phi_{n}$, $\theta_{N_c}$, and $\psi_{n}$ are defined as in \eqref{phinpsin}.
		
Case I: suppose $\mathcal{B}_{2}=\frac{\partial}{\partial \nu} - ik\id$. Using item (1) of \cref{1dX:sol}, we obtain
\[ 
\phi_n=\tfrac{(1+\gl^2)-\frac{\sin(2k\gl)}{2k\gl}(1-\gl^2)^2}{(1+\gl^2)^2 - \cos^2(k\gl)(1-\gl^2)^2},
	\quad
\psi_n=\tfrac{\frac{\sinh(2 k \gl)}{2k\gl}2(\gl^2+1)^2+2(\gl^2-1)}{(\gl^2+1)^2(\cosh(2k \gl)-1)+8\gl^2}.
\]
To obtain an upper bound for $\phi_n$, we note that for all $n \le N_p$ and $k \gl \in (0,\tfrac{\pi}{4}]$
\be \label{sincos}
\begin{aligned}
& 1 -k^2\gl^{2} \le \cos^{2}(k\gl) \le 1 -k^2\gl^{2} + \tfrac{1}{3}k^4\gl^{4}, \quad  1-\tfrac{2}{3}k^2\gl^{2} \le \tfrac{\sin(2k\gl)}{2k\gl}, \quad \text{and}\\
& 1 -k^2\gl^{2} + \tfrac{1}{3}k^4 \gl^{4} \le 1-\tfrac{2}{3}k^2\gl^{2} \le 1.
\end{aligned}
\ee
Moreover, for all $0<k<1$,
\be \label{kless1}
\begin{aligned}
N_p = 0 \quad (\text{i.e., }\gl=1) \quad \text{if} \quad \{\mu_n=n\pi\}_{n \in \NN}; \text{ otherwise, } N_p \text{ does not exist if } \{\mu_n=(n+\tfrac{1}{2})\pi\}_{n \in \NN}.
\end{aligned}
\ee
Now, for $k>0$ and $n \le N_p$,
\begin{align*}
	\phi_n & \le \tfrac{1}{2} \bchi_{\{k<1\}} + \tfrac{(1+\gl^2)^2-\left(1-\frac{2}{3}k^2\lambda_{n}^{2}\right)(1-\gl^2)^2}{(1+\gl^2)^2 - \left(1 -k^2\lambda_{n}^{2} + \frac{1}{3}k^4\lambda_{n}^{4}\right)(1-\gl^2)^2} \bchi_{\{k\ge1,0 < \gl \le \frac{\pi}{4k}\}} + \tfrac{(1+\gl^2)+(1-\gl^2)^2}{(1+\gl^2)^2-(1-\gl^2)^2} \bchi_{\{k\ge1, \frac{\pi}{4k} < \gl \le 1\}}\\
	& \le \tfrac{1}{2} \bchi_{\{k<1\}} + \bchi_{\{k\ge1,0 < \gl \le \frac{\pi}{4k}\}} + \left(\tfrac{1}{2 \gl^2} - \tfrac{1}{4} + \tfrac{1}{4}\gl^2\right) \bchi_{\{k\ge1, \frac{\pi}{4k} < \gl \le 1\}}\\
	& \le \tfrac{1}{2} \bchi_{\{k<1\}} + \bchi_{\{k\ge1,0 < \gl \le \frac{\pi}{4k}\}} + \tfrac{8}{\pi^2}k^2 \bchi_{\{k\ge1, \frac{\pi}{4k} < \gl \le 1\}}\\
	& \le \max\{\tfrac{1}{2},1,\tfrac{8}{\pi^2} \} \max\{k^2,1\} \le \max\{k^2,1\},
\end{align*}
where we respectively used \eqref{kless1} and \eqref{sincos} to obtain the first and second terms of the first inequality.
Next, to obtain an upper bound for $\psi_n$, we note that $\frac{\sinh(x)}{x}\le \cosh(x)$ for all $x\in \R$ and so
\[
\psi_n\le F(\gl,z):=
\tfrac{z2(\gl^2+1)^2+2(\gl^2-1)}{(\gl^2+1)^2(z-1)+8\gl^2}
\quad \mbox{with}\quad z:=\cosh(2k\gl).
\]
Then we have $\frac{dF}{d z}=
\frac{2(\gl^2+1)^2 \gl^2(5-\gl^2)}{((z-1)\gl^4+(2z+6)\gl^2+z-1)^2}$.
For $k >0$ and $0< \gl\le \sqrt{5}$, $F$ is increasing to $\lim_{z\to \infty} F(\gl,z)=2$. Since $\{\mu_n=n\pi\}_{n\in \NN}$ or $\{\mu_n=(n+\frac{1}{2})\pi\}_{n\in \NN}$, we also note that
\begin{equation} \label{lbgl}
\eta:=\sqrt{\tfrac{\pi^2}{4}-1} \le \sqrt{\mu_n^2 - k^2} = k \gl \quad \text{for} \quad 0<k<1 \quad \text{and} \quad \forall n \in \NN \quad \text{such that} \quad \mu_n^2 >k^2.
\end{equation}
Now, for $k>0$ and $n \ge N_e$,
\begin{align*}
\psi_n & \le 2 \bchi_{\{k>0,\gl \le \sqrt{5}\}} +  F(\gl,z) \bchi_{\{k>0,\gl > \sqrt{5}\}} \le 2 \bchi_{\{k>0,\gl \le \sqrt{5}\}} +  \tfrac{2(z+1)}{z-1} \bchi_{\{k>0,\gl > \sqrt{5}\}} \\
& \le 2 \bchi_{\{k>0,\gl \le \sqrt{5}\}} +  \tfrac{2(\cosh(2\sqrt{5})+1)}{\cosh(2\sqrt{5})-1}\bchi_{\{k\ge 1,\gl > \sqrt{5}\}} + \tfrac{2(\cosh(2\eta)+1)}{\cosh(2\eta)-1}\bchi_{\{k<1,\gl > \sqrt{5}\}}\\
& \le \max\left\{2,\tfrac{2(\cosh(2\sqrt{5})+1)}{\cosh(2\sqrt{5})-1}, \tfrac{2(\cosh(2\eta)+1)}{\cosh(2\eta)-1}\right\} \le 3.
\end{align*}
Consequently,
\[
	\max\left\{\max_{0 \le n \le N_p} \phi_{n},\theta_{N_c},\max_{n \ge N_e} \psi_{n}\right\} \le \max\left\{\max\{k^2,1\},\tfrac{2k^2+9}{3k^{2}+12},3\right\} \le 3 \max\{k^{2},1\}.
\]
Plugging in the above estimate back into \eqref{ubg}, applying the Parseval's identity, and finally using \eqref{ineq:ab}, we have \eqref{estim:stability1}.
		
Case II: suppose $\mathcal{B}_{2} = \frac{\partial}{\partial \nu}$. Using item (2) of \cref{1dX:sol} with $\alpha=1$, we obtain
\[
	 \phi_{n}=\tfrac{1+(1-\gl^2)\frac{\sin(2k\gl)}{2k\gl}}{(1-\gl^2)\cos^2(k\gl)+\gl^2}, \quad \psi_{n}=\tfrac{2\left(\frac{\sinh(2k\gl)}{2k\gl}(1+\gl^2)+1\right)}{(\gl^2+1)\cosh(2k\gl) + 1-\gl^2}.
\]
Now, for $k>0$ and $n \le N_p$,
\begin{align*}
	\phi_{n} & \le \bchi_{\{k<1\}} + \tfrac{2-\gl^2}{(1-\gl^2)(1-k^2\gl^2)+\gl^2} \bchi_{\{k\ge 1, \gl \le \min\{\frac{1}{\sqrt{2}},\frac{\pi}{4k}\}\}} + \tfrac{2}{\gl^2} \bchi_{\{k\ge 1, \min\{\frac{1}{\sqrt{2}},\frac{\pi}{4k}\} < \gl \le 1 \}}\\
	& \le \bchi_{\{k<1\}} + \tfrac{2}{1+\gl^2(\gl^2-1)k^2} \bchi_{\{k\ge 1, \gl \le \min\{\frac{1}{\sqrt{2}},\frac{\pi}{4k}\}\}} + 4k^2 \bchi_{\{k\ge 1, \min\{\frac{1}{\sqrt{2}},\frac{\pi}{4k}\} < \gl \le 1\}}\\
	& \le \bchi_{\{k<1\}} + \tfrac{2}{1-\frac{k^2}{4}}\bchi_{\{1\le k < \frac{\pi}{2\sqrt{2}}, \gl \le \frac{1}{\sqrt{2}}\}}
	+ \tfrac{512k^2}{256k^2-16\pi^2k^2+\pi^4} \bchi_{\{k\ge \frac{\pi}{2\sqrt{2}},  \gl \le \frac{\pi}{4k}\}}
	+ 4k^2 \bchi_{\{k\ge 1, \min\{\frac{1}{\sqrt{2}},\frac{\pi}{4k}\} < \gl \le 1\}} \\
	& \le \bchi_{\{k<1\}} + \tfrac{64}{32-\pi^2} \bchi_{\{1\le k < \frac{\pi}{2\sqrt{2}}, \gl \le \frac{1}{\sqrt{2}}\}}
	+ \tfrac{512}{256-16\pi^2} \bchi_{\{k\ge \frac{\pi}{2\sqrt{2}},  \gl \le \frac{\pi}{4k}\}}
	+ 4k^2 \bchi_{\{k\ge 1, \min\{\frac{1}{\sqrt{2}},\frac{\pi}{4k}\} < \gl \le 1\}} \\
	& \le \max\{1, \tfrac{64}{32-\pi^2} ,  \tfrac{512}{256-16\pi^2} , 4\} \max\{k^2,1\} \le 6 \max\{k^2,1\},
\end{align*}
where we respectively used \eqref{kless1} and \eqref{sincos} with $k\lambda_{n} \in (0,\min\{\frac{k}{\sqrt{2}},\frac{\pi}{4}\}]$ to obtain the first and second terms of the first inequality.

Next, to obtain an upper bound for $\psi_n$, we note that $\tfrac{x^2}{\cosh(x)} < \tfrac{5}{4}$ and $\frac{\sinh(x)}{x} \le \cosh(x)$ for all $x\in \R$. Now, for $k>0$ and $n \ge N_e$,
\begin{align*}
\psi_n & =  \tfrac{2\left(\frac{\sinh(2k\gl)}{2k\gl}(1+\gl^2)+1-\gl^2\right)}{\cosh(2k\gl)(1+\gl^2) + 1-\gl^2} + \tfrac{2\gl^2}{\gl^2(\cosh(2k\gl)-1) + 1 + \cosh(2k\gl)} \le 2 + \tfrac{2\gl^2}{\gl^2(\cosh(2k\gl)-1) + 1 + \cosh(2k\gl)}\\
& \le 2 + \tfrac{2}{\cosh(2k\gl)-1} \bchi_{\{k < 1\}} + \tfrac{(2k \gl)^2}{2k^2\cosh(2k\gl)} \bchi_{\{k \ge 1\}}\\
& \le 2 + \tfrac{2}{\cosh(2\eta)-1} \bchi_{\{k < 1\}} + \tfrac{5}{8} \bchi_{\{k \ge 1\}}\\
& \le \max\{2, \tfrac{2}{\cosh(2\eta)-1} ,  \tfrac{5}{8} \} \le 3,
\end{align*}
where we used \eqref{lbgl} to arrive at the second term of the third inequality. Consequently,
\[
	\max\left\{\max_{0 \le n \le N_p} \phi_{n},\theta_{N_c},\max_{n \ge N_e} \psi_{n}\right\} \le \max\left\{6\max\{k^2,1\},2,3\right\} = 6 \max\{k^{2},1\}.
\]
Plugging in the above estimate back into \eqref{ubg}, applying the Parseval's identity, and finally using \eqref{ineq:ab}, we have \eqref{estim:stability1}.
		
Case III: suppose $\mathcal{B}_{2} = \id$. This configuration has been studied in \cite{DLS15}, but we include the proof for the sake of completeness. Using item (2) of \cref{1dX:sol} with $\alpha=0$, we obtain
\[
	\phi_{n}= \tfrac{1-(1-\gl^2)\frac{\sin(2k\gl)}{2k\gl}}{1-(1-\gl^2)\cos^2(k\gl)}, \quad
	 \psi_{n}=\tfrac{2\left(\frac{\sinh(2k\gl)}{2k\gl}(1+\gl^2)-1\right)}{(1+\gl^2)\cosh(2k\gl)-(1-\gl^2)}.
\]
Now, for $k>0$ and $n \le N_p$,
\begin{align*}
	\phi_{n} & \le \bchi_{\{k<1\}} + \tfrac{1-(1-\gl^2)\left(1-\frac{2}{3}k^2\gl^2\right)}{1-(1-\gl^2)\left(1-k^2\gl^2+\frac{1}{3}k^4\gl^{4}\right)} \bchi_{\{k\ge1, \gl \le \frac{\pi}{4k}\}}
	+ \tfrac{2}{\gl^2} \bchi_{\{k\ge1, \frac{\pi}{4k} < \gl \le 1\}}\\
	& \le \bchi_{\{k<1\}} + \bchi_{\{k\ge1, \gl \le \frac{\pi}{4k}\}} + \tfrac{32}{\pi^2}k^2  \bchi_{\{k\ge1, \frac{\pi}{4k} < \gl \le 1\}}\\
	&\le \max\{1,\tfrac{32}{\pi^2} \}\max\{k^2,1\} \le 4\max\{k^2,1\},
\end{align*}
where we respectively used \eqref{kless1} and \eqref{sincos} to obtain the first and second terms of the first inequality. Next, for $k>0$ and $n \ge N_e$, we have
\[
\psi_{n} \le \tfrac{2\left(\frac{\sinh(2k\gl)}{2k\gl}(1+\gl^2)-1\right)}{(1+\gl^2)\cosh(2k\gl)-1} \le 2.
\]
Consequently,
\[
	\max\left\{\max_{0 \le n \le N_p} \phi_{n},\theta_{N_c},\max_{n \ge N_e} \psi_{n}\right\} \le \max\left\{ 4 \max\{k^2,1\},\tfrac{2k^2+3}{3k^2+3},2\right\} = 4 \max\{k^{2},1\}.
\]
Plugging in the above estimate back into \eqref{ubg}, applying the Parseval's identity, and finally using \eqref{ineq:ab}, we have \eqref{estim:stability1}.
\end{proof}

\begin{proof}[Proof of \cref{thm:stability2}]
	We shall only focus on items (2) and (3), since the proof of item (1) is identical to the proof of \cref{thm:stability1} (Case I). Given the boundary assumptions, the solution $u$ can be expressed as
	\begin{equation*} 
		u = \sum_{n=0}^{\infty} \wh{g_{2}}(n) X_{n}(x) Y_n(y) \quad \text{with} \quad \wh{g_{2}}(n):=\int_{\Gamma_2}g_{2}(y)Y_{n}(y) dy,
	\end{equation*}
	where $\{X_{n}\}_{n \in \NN}$ are stated in \cref{1dX:sol} and $\{Y_{n}\}_{n \in \N_{0}}$ are stated in \eqref{1dY:sol}.
	
	Recall that  $\gl:=\sqrt{\left|1-\frac{\mu_{n}^{2}}{k^2}\right|}$ and observe that $\|Y_n'\|_{0,\mathcal{I}}=\mu_n$. By \eqref{Zn}, since $\{Y_{n}\}_{n \in \NN}$ and $\{Y'_{n}\}_{n \in \NN}$ are orthogonal systems in $L^{2}(\mathcal{I})$, we deduce that
	\be \label{normXY}
	\begin{aligned}
		\|\nabla u \|^{2}_{0,\Omega} & + k^{2} \|u\|^{2}_{0,\Omega} = \left\|\sum_{n=0}^{\infty}\wh{g_{2}}(n)X'_{n}Y_{n}\right\|^{2}_{0,\Omega} + \left\|\sum_{n=0}^{\infty}\wh{g_{2}}(n)X_{n}Y'_{n}\right\|^{2}_{0,\Omega} + k^{2} \left\|\sum_{n=0}^{\infty}\wh{g_{2}}(n)X_{n}Y_{n}\right\|^{2}_{0,\Omega}\\
		& = \sum_{n=0}^{\infty} |\wh{g_{2}}(n)|^{2} \|X'_{n}\|^{2}_{0,\mathcal{I}} + \sum_{n=0}^{\infty} |\wh{g_{2}}(n) \mu_{n}|^{2} \|X_{n}\|^{2}_{0,\mathcal{I}} + k^{2} \sum_{n=0}^{\infty} |\wh{g_{2}}(n)|^{2} \|X_{n}\|^{2}_{0,\mathcal{I}},
	\end{aligned}
	\ee
	Regrouping the terms, we have
 	\begin{equation}\label{ubginDN}
		\sum_{n=0}^{\infty} |\wh{g_{2}}(n)|^{2} (\|X'_{n}\|^{2}_{0,\mathcal{I}} + (\mu_{n}^{2} + k^{2})\|X_{n}\|^{2}_{0,\mathcal{I}})
		\le
		\max\left\{\max_{0 \le n \le N_p} \phi_{n}, \theta_{N_c},\max_{n \ge N_e} \psi_{n}\right\}\sum_{n=0}^{\infty} |\wh{g_{2}}(n)|^{2},
	\end{equation}
	where $\phi_{n}$, $\theta_{N_c}$, and $\psi_{n}$ are defined as in \eqref{phinpsin}.
		
	Item (2): suppose $\mathcal{B}_{2}=\frac{\partial}{\partial \nu}$. Using item (3) of \cref{1dX:sol} with $\alpha=1$, we obtain
	\[ 
	 \phi_{n}=\tfrac{(1+\lambda_{n}^2) - \frac{\sin(2k\lambda_n)}{2k\lambda_n}(1-\lambda_n^2)^2}{\lambda_{n}^{2}((1 -\lambda_n^2)\cos^2(k\lambda_n) + \lambda_n^2)},
		\quad
	\psi_{n}= \tfrac{2\left((1 + \gl^2)^2 \frac{\sinh(2k\gl)}{2k\gl} + (\gl^{2}-1)\right)}{\gl^2\left((\gl^2 + 1)(\cosh(2k\gl)-1)+2\right)}.
	\]
	First, we note that for $k\ge 1$,
	\[
	\tfrac{d}{d\gl}\Big(2k^2\gl^4 - (4k^2+3)\gl^2+2k^2+9\Big) = 2\gl(4(\gl^2-1)k^2-3) \le 0, \quad \forall \gl \in (0,\sqrt{\tfrac{3}{4k^2}+1}].
	\]
	Now, for $k>0$ and $n \le N_p$,
	\begin{align*}
		\phi_{n} & \le 2\bchi_{\{k<1\}} + \tfrac{(1+\gl^2)-\left(1-\tfrac{2}{3}k^2\gl^2\right)(1-\gl^2)^2}{\gl^2((1-\gl^2)(1-k^2\gl^2)+\gl^2)}\bchi_{\{k\ge1,\gl \le \frac{\pi}{4k}\}} + \tfrac{(1+\gl^2)+(1-\gl^2)^2}{\gl^4}\bchi_{\{k\ge1,\frac{\pi}{4k}<\gl\le 1\}},\\
		& \le 2\bchi_{\{k<1\}} + \tfrac{2k^2\gl^4-(4k^2+3)\gl^2+2k^2+9}{3+3\gl^2(\gl^2-1)k^2} \bchi_{\{k\ge1,\gl \le \frac{\pi}{4k}\}} +
		\left(\tfrac{512}{\pi^4}k^4 - \tfrac{16}{\pi^2}k^2+1\right)
		 \bchi_{\{k\ge1,\frac{\pi}{4k}<\gl\le 1\}}\\
		& \le 2\bchi_{\{k<1\}} +  \tfrac{2k^2+9}{3(1-\gl^2k^2)} \bchi_{\{k\ge1,\gl \le \frac{\pi}{4k}\}} +
		\left(\tfrac{512}{\pi^4}k^4 - \tfrac{16}{\pi^2}k^2+1\right)
		 \bchi_{\{k\ge1,\frac{\pi}{4k}<\gl\le 1\}}\\
		& \le 2\bchi_{\{k<1\}} +  \tfrac{2k^2+9}{3(1-\frac{\pi^2}{16})} \bchi_{\{k\ge1,\gl \le \frac{\pi}{4k}\}} +
		\left(\tfrac{512}{\pi^4}k^4 - \tfrac{16}{\pi^2}k^2+1\right)
		 \bchi_{\{k\ge1,\frac{\pi}{4k}<\gl\le 1\}} \\
		& \le \max\left\{2, \tfrac{11}{3(1-\frac{\pi^2}{16})}, \tfrac{512}{\pi^4} - \tfrac{16}{\pi^2}+1 \right\} \max\{k^4,1\} \le 10 \max\{k^4,1\},
	\end{align*}
	where we respectively used \eqref{kless1} and \eqref{sincos} to obtain the first and second terms of the first inequality. Next, we note that for all $n\ge N_e$ and $k\gl \in (0,1]$,
	\be \label{sinhcosh}
	\tfrac{\sinh(2k\gl)}{2k\gl} \le 1+\tfrac{4}{3}k^2\gl^2 \quad \text{and} \quad 1+2k^2\gl^2\le \cosh(2k\gl).
	\ee
	Now, for $k>0$ and $n \ge N_e$,
	\begin{align*}
		\psi_{n} & \le 	 
		 2\left(\tfrac{(1+\gl^2)\frac{\sinh(2k\gl)}{2k\gl}+1}{\gl^2(\cosh(2k\gl)-1)}\right) \bchi_{\{k<1\} \cup \{k \ge 1, \gl > \frac{1}{k}\}}
		+ \tfrac{(1+\gl^2)^2\left(1+\frac{4}{3}k^2\gl^2\right)+(\gl^2-1)}{\gl^2} \bchi_{\{k \ge 1, \gl \le \frac{1}{k}\}}\\
		 & \le 2\left(\tfrac{k^2}{k^2\gl^2}+1\right)\left(\tfrac{\cosh(2k\gl)+1}{\cosh(2k\gl)-1}\right) \bchi_{\{k<1\} \cup \{k \ge 1, \gl > \frac{1}{k}\}} + \left(\gl^2+3+\tfrac{4}{3}k^2\gl^4+\tfrac{8}{3}k^2\gl^2+\tfrac{4}{3}k^2\right) \bchi_{\{k \ge 1, \gl \le \frac{1}{k}\}}\\
		 & \le 2\left(\tfrac{1}{\eta^2}+1\right) \left(\tfrac{\cosh(2\eta)+1}{\cosh(2\eta)-1}\right) \bchi_{\{k<1\}} + \left(\tfrac{7}{3k^2}+\tfrac{17}{3}+\tfrac{4}{3}k^2\right) \bchi_{\{k \ge 1, \gl \le \frac{1}{k}\}} + \left(2k^2+2\right) \left(\tfrac{\cosh(2)+1}{\cosh(2)-1}\right)\bchi_{\{k \ge 1, \gl > \frac{1}{k}\}}\\
		 & \le \max\left\{ 2\left(\tfrac{1}{\eta^2}+1\right) \left(\tfrac{\cosh(2\eta)+1}{\cosh(2\eta)-1}\right), \tfrac{28}{3},  4 \left(\tfrac{\cosh(2)+1}{\cosh(2)-1}\right)  \right\} \max\{k^2,1\} \le 10 \max\{k^2,1\},
	\end{align*}
	where we used \eqref{sinhcosh} to arrive at the second term of the first inequality and \eqref{lbgl} to arrive at the first term of the third inequality. Consequently,
	\[
		\max\left\{\max_{0 \le n \le N_p} \phi_{n},\theta_{N_c},\max_{n \ge N_e} \psi_{n}\right\} \le \max\left\{10\max\{k^4,1\},\tfrac{2}{3}k^2+3,10\max\{k^2,1\}\right\} = 10 \max\{k^{4},1\}.
	\]
	Plugging in the above estimate back into \eqref{ubginDN}, applying the Parseval's identity, and finally using \eqref{ineq:ab}, we have \eqref{Ck:B4inneum}.
		
	Item (3): suppose $\mathcal{B}_{2}=\id$. First, we note that for $k\ge 1$,
	\[
	 \tfrac{d}{d\gl}\Big(
3+\gl^2(\gl^2-1)k^4+3(1-\gl^2)k^2\Big) = 2\gl ((2\gl^2-1)k^4-3k^2)\le 0, \quad \forall \gl \in (0,\sqrt{\tfrac{3}{2k^2}+\tfrac{1}{2}}].
	\]
	By item (3) of \cref{1dX:sol} with $\alpha=0$, we obtain for $k>0$ and $n \le N_p$,
	\begin{align*}
	\phi_{n} & =\tfrac{k^{2}\left((1+\lambda_{n}^{2}) - (1-\lambda_{n}^2)^2 \frac{\sin(2k\lambda_n)}{2k\lambda_n}\right)}
	 {1-(1-\lambda_{n}^{2})\cos^{2}(k\lambda_{n})}
	\le 2\bchi_{\{k<1\}} + \tfrac{k^2((1+\gl^2)+(1-\gl^2)^2)}{\gl^2}\bchi_{\{k\ge1,\frac{\pi}{4k} < \gl \le 1\}} \\ & \quad + k^2\left(\tfrac{(1-\gl^2)\left(1-(1-\gl^2)\left(1-\frac{2}{3}k^2\gl^2\right)\right)}{1-(1-\gl^2)\left(1-k^2\gl^2+\frac{1}{3}k^4\gl^4\right)} +  \tfrac{2\gl^2}{1-(1-\gl^2)\left(1-k^2\gl^2+\frac{1}{3}k^4\gl^4\right)}\right)\bchi_{\{k\ge1,\gl \le \frac{\pi}{4k}\}} \\
	& \le 2\bchi_{\{k<1\}} + k^2 \left((1-\gl^2)+ \tfrac{6}{3+\gl^2(\gl^2-1)k^4+3(1-\gl^2)k^2}\right) \bchi_{\{k\ge1,\gl \le \frac{\pi}{4k}\}} + k^2\left(\gl^2-1+\tfrac{2}{\gl^2}\right)\bchi_{\{k\ge1,\frac{\pi}{4k} < \gl \le 1\}}\\
	& \le 2\bchi_{\{k<1\}} + k^2 \left(1+\tfrac{1536}{(768-16\pi^2)k^2+(768+\pi^4-48\pi^2)}\right)\bchi_{\{k\ge1,\gl \le \frac{\pi}{4k}\}} + \tfrac{32}{\pi^2}k^4\bchi_{\{k\ge1,\frac{\pi}{4k} < \gl \le 1\}}\\
	& \le \max\left\{2, 1+\tfrac{1536}{(768-16\pi^2)+(768+\pi^4-48\pi^2)}, \tfrac{32}{\pi^2} \right\}\max\{k^4,1\} \le 4\max\{k^4,1\},
	\end{align*}
	where we respectively used \eqref{kless1} and \eqref{sincos} to obtain the first and second terms of the first inequality, and substitute $\gl=\tfrac{\pi}{4k}$ into the second term of the third inequality. Next, for $n \ge N_e$ and $\lambda_{n} \in (0,\infty)$, we have
	\[ 
		k^2 \|X_{n}\|^{2}_{0,\mathcal{I}}
		= k^2\tfrac{\frac{\sinh(2k\gl)}{2k\gl}(1+\gl^2)-(1-\gl^2)}
		 {\cosh(2k\gl)(1+\gl^2)-(1-\gl^2)}
		\le k^2.
	\]
	Next, we note that for all $\gl \in \mathbb{R}$
	\[
	 \tfrac{\frac{\sinh(2k\gl)}{2k\gl}-1}{\cosh(2k\gl)-1} \le \lim_{k\gl \rightarrow 0}\tfrac{\frac{\sinh(2k\gl)}{2k\gl}-1}{\cosh(2k\gl)-1} \le \tfrac{1}{3}, \quad \tfrac{\gl(\frac{3}{2}+\gl^2)}{(1+\gl^2)^{\frac{3}{2}}} \le \lim_{\gl \rightarrow \infty} \tfrac{\gl(\frac{3}{2}+\gl^2)}{(1+\gl^2)^{\frac{3}{2}}}\le 1, \quad \text{and} \quad
	 \tfrac{\gl^{2}}{(1+\gl^2)^{\frac{3}{2}}}\le \tfrac{2\sqrt{3}}{9}.
	\]
	By item (3) of \cref{1dX:sol} with $\alpha=0$, we obtain for $k>0$ and $n \ge N_e$
		\begin{align*}
			& \tfrac{\|X'_{n}\|^{2}_{0,\mathcal{I}} + \mu_{n}^{2} \|X_{n}\|^{2}_{0,\mathcal{I}}}{\mu_{n}}
			= \tfrac{k((1+ 3 \gl^2 + 2\gl^4)\frac{\sinh(2k\gl)}{2k\gl} +\gl^2 -1)}{(1 + \gl^2)^{\frac{3}{2}} \cosh(2k\gl) - (1-\gl^2)(1+\gl^2)^{\frac{1}{2}}} \\
			& \quad \le \tfrac{k((1+ 3 \gl^2 + 2\gl^4)\frac{\sinh(2k\gl)}{2k\gl} +\gl^2 + 1)}{(1 + \gl^2)^{\frac{3}{2}} (\cosh(2k\gl) - 1)} \bchi_{\{k<1\}}
			+ \tfrac{k((1+ 3 \gl^2 + 2\gl^4)\frac{\sinh(2k\gl)}{2k\gl} +\gl^2 -1)}
			{(1 + \gl^2)^{\frac{1}{2}} ((1 + \gl^2)\cosh(2k\gl) - (1-\gl^2))} \bchi_{\{k\ge1,\gl \le \frac{1}{k}\} \cup \{k\ge1,\gl > \frac{1}{k}\}}\\
			& \quad \le \left(\tfrac{(1+3\gl^2+2\gl^4)\sinh(2k\gl)}{2\gl^4(\cosh(2k\gl)-1)} + \tfrac{1}{(1+\gl^2)^{\frac{1}{2}}(\cosh(2k\gl)-1)}\right)\bchi_{\{k<1\}}
			+ \tfrac{k((1+ 3 \gl^2 + 2\gl^4)\left(1+\frac{4}{3}k^2\gl^2\right) +\gl^2 - 1)}{(1 + \gl^2)^{\frac{1}{2}} ((1 + \gl^2)(1 + 2k^2\gl^2) - (1-\gl^2))} \bchi_{\{k\ge1,\gl \le \frac{1}{k}\}}\\
			& \qquad + \left(\tfrac{\left(\frac{\sinh(2k\gl)}{2k\gl}-1\right)}{(1+\gl^2)^{\frac{3}{2}}(\cosh(2k\gl)-1)} k
			+ \left(\tfrac{\gl\left(\frac{3}{2}+\gl^2\right)}{(1+\gl^2)^{\frac{3}{2}}}\right) \left(\tfrac{\sinh(2k\gl)}{\cosh(2k\gl)-1}\right) + \tfrac{\gl^{2} }{(1+\gl^2)^{\frac{3}{2}} (\cosh(2k\gl)-1)} k
			\right) \bchi_{\{k\ge1,\gl > \frac{1}{k}\}}\\
			& \quad \le
			 \left(\left(\tfrac{1}{2\eta^4} +\tfrac{3}{2\eta^2} +1\right)\left(\tfrac{\sinh(2\eta)}{\cosh(2\eta)-1}\right) + \tfrac{1}{\cosh(2\eta)-1}\right) \bchi_{\{k<1\}} +
			 \tfrac{(4\gl^4+6\gl^2+2)k^3+3k(\gl^2+2)}{3+(3\gl^2+3)k^2} \bchi_{\{k\ge1,\gl \le \frac{1}{k}\}}\\
			& \qquad + \left(\tfrac{1}{3}k + \tfrac{\sinh(2)}{\cosh(2)-1} + \tfrac{2\sqrt{3}}{9(\cosh(2)-1)}k\right) \bchi_{\{k\ge1,\gl > \frac{1}{k}\}}\\
			& \quad \le
			\tfrac{(2\eta^4 + 3\eta +1)\sinh(2\eta)+2\eta^4}{2\eta^4 (\cosh(2\eta)-1)} \bchi_{\{k<1\}} +
			 \tfrac{(4\gl^4+6\gl^2+2)k^3+3k(\gl^2+2)}{3k^2} \bchi_{\{k\ge1,\gl \le \frac{1}{k}\}} +  \tfrac{3(\cosh(2)-1)+9\sinh(2)+2\sqrt{3}}{9(\cosh(2)-1)}k \bchi_{\{k\ge1,\gl > \frac{1}{k}\}}\\
			& \quad \le
			\tfrac{(2\eta^4 + 3\eta +1)\sinh(2\eta)+2\eta^4}{2\eta^4 (\cosh(2\eta)-1)} \bchi_{\{k<1\}} +
			 \left(\tfrac{7}{3k^3}+\tfrac{4}{k}+\tfrac{2}{3}k\right) \bchi_{\{k\ge1,\gl \le \frac{1}{k}\}} +  \tfrac{3(\cosh(2)-1)+9\sinh(2)+2\sqrt{3}}{9(\cosh(2)-1)}k \bchi_{\{k\ge1,\gl > \frac{1}{k}\}}\\
			& \quad \le \max\left\{ \tfrac{(2\eta^4 + 3\eta +1)\sinh(2\eta)+2\eta^4}{2\eta^4 (\cosh(2\eta)-1)} ,7, \tfrac{3(\cosh(2)-1)+9\sinh(2)+2\sqrt{3}}{9(\cosh(2)-1)} \right\} \max\{k,1\} \le 7 \max\{k,1\},
		\end{align*}
		where we used \eqref{sinhcosh} to arrive at the last term of the second inequality and \eqref{lbgl} to arrive at the first term of the third inequality. Continuing from \eqref{normXY}, we have
		{\small
		\begin{align*}
			&
			\sum_{n=0}^{\infty} |\wh{g_{2}}(n)|^{2} \|X'_{n}\|^{2}_{0,\mathcal{I}} + \sum_{n=0}^{\infty} |\wh{g_{2}}(n) \mu_{n}|^{2} \|X_{n}\|^{2}_{0,\mathcal{I}} + k^{2} \sum_{n=0}^{\infty} |\wh{g_{2}}(n)|^{2} \|X_{n}\|^{2}_{0,\mathcal{I}}\\
			&
			\quad = \sum_{n=0}^{N_e-1} |\wh{g_{2}}(n)|^{2} \left(\|X'_{n}\|^{2}_{0,\mathcal{I}} + (\mu_{n}^{2}+k^2) \|X_{n}\|^{2}_{0,\mathcal{I}}\right) + \sum_{n=N_e}^{\infty} |\wh{g_{2}}(n)|^{2} \mu_{n} \left(\tfrac{\|X'_{n}\|^{2}_{0,\mathcal{I}} + \mu_{n}^{2} \|X_{n}\|^{2}_{0,\mathcal{I}}}{\mu_{n}}\right) +  \sum_{n=N_e}^{\infty} k^{2}|\wh{g_{2}}(n)|^{2}\|X_{n}\|^{2}_{0,\mathcal{I}}\\
			&
			\quad \le \max\left\{\max_{0\le n \le N_p} \phi_{n},\tfrac{k^{2}(2k^{2}+9)}{3(k^{2}+1)}\right\} \sum_{n=0}^{N_e-1} |\wh{g_{2}}(n)|^{2} + 7 \max\{k,1\} \sum_{n=N_e}^{\infty} |\wh{g_{2}}(n)|^{2} \mu_n + k^{2}\sum_{n=N_e}^{\infty} |\wh{g_{2}}(n)|^{2} \\
			&
			\quad \le 4\max\{k^4,1\} \sum_{n=0}^{\infty} |\wh{g_{2}}(n)|^{2} +7 \max\{k,1\} \sum_{n=0}^{\infty} |\wh{g_{2}}(n)|^{2} \mu_n \\
			&
			\quad \le 7 \left(\max\{k^4,1\} \|g_{2}\|^{2}_{0,\Gamma_{2}} + \max\{k,1\} \|g_{2}\|^{2}_{\mathcal{Z}^{\frac{1}{2}}(\Gamma_{2})}\right),
		\end{align*}
		}	
		where we used our assumptions that $\mathcal{Z}^{\frac{1}{2}}(\Gamma_{2})$ and applied the Parseval's identity to arrive at the last line. Finally by \eqref{ineq:ab}, the stability estimate in \eqref{Ck:B4indiric} is proved.
	\end{proof}

\begin{proof}[Proof of \cref{thm:stabilityf}]
	All three cases start the same way. Letting $v=u$ in \eqref{weakform}, we have
	\be \label{weakuu}
	\| \nabla u\|^{2}_{0,\Omega} - k^2\|u\|_{0,\Omega}^2 - i k \|u\|_{0,\Gamma_R}^2 = \langle f, u\rangle_{\Omega},
	\ee
	where $\Gamma_R=\Gamma_2 \cup \Gamma_4$ if $\mathcal{B}_2=\frac{\partial}{\partial \nu}-ik\id$; otherwise, $\Gamma_R=\Gamma_4$ if $\mathcal{B}_2 \in \{\id, \frac{\partial}{\partial \nu}\}$. Separately considering the real and imaginary parts of \eqref{weakuu}, and applying the Cauchy-Schwarz inequality, we have
	\be \label{ReImmodulus}
	\| \nabla u \|^{2}_{0,\Omega} \le k^{2} \|u\|^{2}_{0,\Omega} + \|f\|_{0,\Omega}\|u\|_{0,\Omega}, \quad k\|u\|^{2}_{0,\Gamma_{R}} \le \|f\|_{0,\Omega}\|u\|_{0,\Omega}.
	\ee
	It is known in \cite[Proposition 2.1]{CF06} that for $u \in H^{2}(\Omega)$ and $\textbf{z} \in (C^{1}(\overline{\Omega}))^{2}$, the following identity holds
	\begin{equation}\label{rellich}
		2 \Re \int_{\Omega} \Delta u (\textbf{z} \cdot \nabla \overline{u}) = 2 \Re \int_{\partial \Omega} \frac{\partial u}{\partial \nu} (\textbf{z} \cdot \nabla \overline{u}) -2 \Re \int_{\Omega} \nabla u \cdot (\nabla \overline{u} \cdot \nabla)\textbf{z} + \int_{\Omega} (\nabla \cdot \textbf{z}) |\nabla u|^{2} - \int_{\partial \Omega} \textbf{z}\cdot \textbf{n} |\nabla u|^{2}.
	\end{equation}
	Take $\textbf{z}=(x-1,0)$. Since $\Delta u= - f- k^{2}u$, $2\Re(u \overline{u}_{x})=(|u|^{2})_{x}$, by applying integration by parts to the left-hand side of \eqref{rellich}, we have
	\[
	2 \Re \int_{\Omega} \Delta u (\textbf{z} \cdot \nabla \overline{u}) = -2\Re \int_{\Omega} (x-1) f \overline{u}_{x} + k^{2} (\|u\|^{2}_{0,\Omega} -\|u\|^{2}_{0,\Gamma_{4}}).
	\]
	Note that $\frac{\partial u}{\partial \nu} =iku$ on $\Gamma_{4}$. If $\mathcal{B}_{1}u=u=0$ on $\Gamma_1$, then $u_{x}(x,0)=(u(x,0))_{x}=0$. Similarly, if $\mathcal{B}_{3}u=u=0$ on $\Gamma_3$, then $u_{x}(x,1)=(u(x,1))_{x}=0$. Therefore, we have
	$2 \Re \int_{\partial \Omega} \frac{\partial u}{\partial \nu} (\textbf{z} \cdot \nabla \overline{u})=2k^{2} \|u\|^{2}_{0,\Gamma_{4}}$ for any $\mathcal{B}_1,\mathcal{B}_3 \in \{\id, \frac{\partial}{\partial \nu}\}$. Hence, fully expanding \eqref{rellich}, we have
	\begin{align*}
		-2\Re \int_{\Omega} (x-1) f \overline{u}_{x} + k^{2} (\|u\|^{2}_{0,\Omega} -\|u\|^{2}_{0,\Gamma_{4}}) & = 2k^{2} \|u\|^{2}_{0,\Gamma_{4}} - 2\|u_{x}\|^{2}_{0,\Omega} + \|\nabla u\|^{2}_{0,\Omega} - \|\nabla u\|^{2}_{0,\Gamma_{4}}.\\
		& = 2k^{2} \|u\|^{2}_{0,\Gamma_{4}} - \|u_{x}\|^{2}_{0,\Omega} + \|u_{y}\|^{2}_{0,\Omega} - \|\nabla u\|^{2}_{0,\Gamma_{4}},
	\end{align*}
	from which, after using \eqref{weakuu} to replace $\|u_{y}\|^{2}_{0,\Omega}$,  we obtain
	\begin{equation}\label{calc:f}
		2\|u_{x}\|^{2}_{0,\Omega} + \|u_{y}\|^{2}_{0,\Gamma_{4}}= \langle f,u\rangle_{\Omega} + ik\|u\|^{2}_{0,\Gamma_{R}} + 2\Re\langle (x-1)f,u_{x}\rangle_{\Omega} + 2k^{2} \|u\|_{0,\Gamma_{4}}^2.
	\end{equation}
	Taking the real part of \eqref{calc:f}, we have
	\be \label{f:uxnorm}
	2\|u_{x}\|^{2}_{0,\Omega} \le \|f\|_{0,\Omega} \|u\|_{0,\Omega} + 2 \|f\|_{0,\Omega} \|u_{x}\|_{0,\Omega}+ 2k^{2}\|u\|^{2}_{0,\Gamma_{4}}
	\ee

	Suppose $\mathcal{B}_2=\id$. Since $\|u\|^{2}_{0,\Omega} \le \|u_{x}\|^{2}_{0,\Omega}$ (which is proved by noting that $u(x,y)=-\int_{x}^{1}u_{x}(s,y)ds$ and then estimating an upper bound), \eqref{f:uxnorm} and the second inequality of \eqref{ReImmodulus} yield $\|u_{x}\|_{0,\Omega} \le \frac{1}{2}(3+2k)\|f\|_{0,\Omega}$. I.e., $\|u\|_{0,\Omega} \le \|u_{x}\|_{0,\Omega} \le \frac{1}{2}(3+2k) \|f\|_{0,\Omega}$. So, by the first inequality of \eqref{ReImmodulus}, we have
		\[
		\| \nabla u \|^{2}_{0,\Omega} + k^{2} \|u\|^{2}_{0,\Omega}\le 2k^{2} \|u\|^{2}_{0,\Omega} + \|f\|_{0,\Omega}\|u\|_{0,\Omega} \le \left(\tfrac{1}{2}k^{2}(3+2k)^2 + \tfrac{1}{2}(3+2k)\right) \|f\|_{0,\Omega}^{2} \le 15 \max\{k^{4},1\} \|f\|_{0,\Omega}^{2}.
		\]
	Finally, by using \eqref{ineq:ab}, we obtain \eqref{estim:withf:nd}.
	
	Suppose $\mathcal{B}_2 = \frac{\partial}{\partial \nu}$. We have $\|u\|_{0,\Omega}^2 \le 2(\|u\|^2_{0,\Gamma_4} + \|u_x\|^2_{0,\Omega}$) (which is proved by noting that $u(x,y)=u(0,y) + \int_{0}^x u_x(s,y)ds$ and then estimating an upper bound). By the second inequality of \eqref{ReImmodulus}, we have
	$
	\|u\|^2_{0,\Omega} - 2k^{-1} \|f\|_{0,\Omega} \|u\|_{0,\Omega} -2 \|u_x\|^2_{0,\Omega} \le 0.
	$
	This implies that
	\be \label{unorm:neum}
	\|u\|_{0,\Omega} \le k^{-1} \|f\|_{0,\Omega} + \tfrac{1}{2} \sqrt{4k^{-2} \|f\|^2_{0,\Omega} + 8 \|u_x\|^2_{0,\Omega}} \le 2k^{-1} \|f\|_{0,\Omega} + \sqrt{2} \|u_x\|_{0,\Omega}.
	\ee
	By \eqref{f:uxnorm}, the second inequality of \eqref{ReImmodulus}, and \eqref{unorm:neum}, we have
	\[
	2\|u_x\|^2_{0,\Omega} \le (2k+1) \|f\|_{0,\Omega} \|u\|_{0,\Omega} + 2\|f\|_{0,\Omega} \|u_x\|_{0,\Omega} \le 2(2+k^{-1}) \|f\|^2_{0,\Omega} + (2\sqrt{2} k + \sqrt{2} + 2) \|f\|_{0,\Omega}\|u_x\|_{0,\Omega}.
	\]
	I.e.,
	$
	2\|u_x\|^2_{0,\Omega} - (2\sqrt{2} k + \sqrt{2} + 2) \|f\|_{0,\Omega}\|u_x\|_{0,\Omega} - 2(2+k^{-1}) \|f\|^2_{0,\Omega} \le 0.
	$ This implies that
	\be \label{uxnorm:neum}
	\begin{aligned}
	\|u_x\|_{0,\Omega} & \le \tfrac{1}{4}(2\sqrt{2} k + \sqrt{2} + 2) \|f\|_{0,\Omega} + \tfrac{1}{4} \sqrt{(2\sqrt{2} k + \sqrt{2} + 2)^2 + 16(2+k^{-1})} \|f\|_{0,\Omega}\\
	& \le (\tfrac{1}{2}(2\sqrt{2} k + \sqrt{2} + 2) + (2+k^{-1})^{1/2})\|f\|_{0,\Omega} \le (\sqrt{2} k + \tfrac{3}{2}\sqrt{2}+1+k^{-1/2})\|f\|_{0,\Omega}.
	\end{aligned}
	\ee
	By the first inequality of \eqref{ReImmodulus}, \eqref{unorm:neum}, and \eqref{uxnorm:neum}, we have
	\[
	\begin{aligned}
		& \| \nabla u \|^{2}_{0,\Omega} + k^{2} \|u\|^{2}_{0,\Omega}  \le 2k^{2} \|u\|^{2}_{0,\Omega} + \|f\|_{0,\Omega}\|u\|_{0,\Omega} \\ 
		& \quad \le 4k^2 (4k^{-2} \|f\|_{0,\Omega}^2 + 2 \|u_x\|_{0,\Omega}^2) + 2k^{-1} \|f\|_{0,\Omega}^2 + \sqrt{2} \|f\|_{0,\Omega}\|u_x\|_{0,\Omega}\\
		& \quad \le 4k^2 (4k^{-2} + 2 (\sqrt{2} k + \tfrac{3}{2}\sqrt{2}+1+k^{-1/2})^2)\|f\|_{0,\Omega}^2 + 2k^{-1} \|f\|_{0,\Omega}^2 + \sqrt{2} (\sqrt{2} k + \tfrac{3}{2}\sqrt{2}+1+k^{-1/2}) \|f\|_{0,\Omega}^2\\
		& \quad \le (155+82\sqrt{2}) \max\{k^4,k^{-1}\} \|f\|_{0,\Omega}^2  \le 271 \max\{k^4,k^{-1}\} \|f\|_{0,\Omega}^2.
	\end{aligned}
	\]
	Finally, by using \eqref{ineq:ab}, we obtain \eqref{estim:withf:rn}.
	
	Suppose $\mathcal{B}_2=\frac{\partial}{\partial \nu}-ik\id$. Keeping in mind that the second inequality of \eqref{ReImmodulus} implies $k\|u\|^{2}_{0,\Gamma_{4}} \le k\|u\|^{2}_{0,\Gamma_{2} \cup \Gamma_{4}}\le \|f\|_{0,\Omega}\|u\|_{0,\Omega}$, the proof of this case is identical to the case where $\mathcal{B}_2=\frac{\partial}{\partial \nu}$.
\end{proof}

In order to prove \cref{traceL2} and \cref{thm:stability3}, we first prove two auxiliary results stated in \cref{lem:tYn,normsaux} below.

\begin{lemma}\label{lem:tYn}
	Let $\tilde{\mu}_n, n\in \NN$ be given in \eqref{allcases} and
	define $\tilde{\lambda}_n:=\sqrt{\left|1-\tfrac{\tilde{\mu}_n^2}{k^2}\right|}$ for $n\in \NN$.
	For $\bc_1=\bc_3$,
	\be \label{gl:lower}
	\inf_{j\in \Z} |\katgl -j\pi|\ge \tfrac{\frac{1}{8}\pi}{1+2\pi^{-1} \katgl},\qquad \forall\; n\in \NN,
	\ee
	and for $\bc_1\ne \bc_3$,
	\be \label{gl:lower:2}
	\inf_{j\in \Z} |\katgl -(j+\tfrac{1}{2})\pi|\ge \tfrac{\frac{1}{8}\pi}{1+2\pi^{-1} \katgl},\qquad \forall\; n\in \NN.
	\ee
\end{lemma}

	\begin{proof}
		We first consider $\bc_1=\bc_3$. Let $j$ be the unique integer such that
		$j\le \frac{\katgl}{\pi}<j+1$.
Then it is obvious that $\inf_{m\in \Z}|\katgl-m\pi|=
\min(|\katgl-j\pi|, |\katgl-(j+1)\pi)|)$.
		If $d_0\in [0,\frac{1}{8}\pi^2]$, then $-\frac{1}{4}\pi^2 \pm d_0\in [-\frac{3}{8}\pi^2,-\frac{1}{8}\pi^2]$ and $\tilde{\mu}_n=(n+\frac{1}{2})\pi$ by \eqref{allcases}.
		By the definition of $d_0$, we have $k^2=m\pi^2 \pm d_0$ for some $m\in \Z$. By $\tilde{\mu}_n=(n+\tfrac{1}{2})\pi$ in \eqref{allcases} and $(\katgl)^2=|k^2-\tilde{\mu}_n^2|=|(m-n^2-n)\pi^2+(-\frac{1}{4}\pi^2\pm d_0)|$, we have
		\[
		 |\katgl-j\pi|=\tfrac{|(\katgl)^2-j^2\pi^2|}
		{\katgl+j\pi}
		=\tfrac{|(N-j^2)\pi^2\pm (-\frac{1}{4}\pi^2\pm d_0)|}{\katgl+j\pi}
		\ge \tfrac{\frac{1}{8}\pi^2}{\katgl+j\pi}
		\ge \tfrac{\frac{1}{8}\pi}{2\pi^{-1}\katgl+1},
		\]
		where $N:=m-n^2-n$ for $k^2\ge \tilde{\mu}_n^2$ or $N:=n^2+n-m$ for $k^2<\tilde{\mu}_n^2$,
		and
		\[	 |\katgl-(j+1)\pi|=\tfrac{|(\katgl)^2-(j+1)^2\pi^2|}
		 {\katgl+(j+1)\pi}
		 =\tfrac{|((j+1)^2-N)\pi^2\pm (-\frac{1}{4}\pi^2\pm d_0)|}{\katgl+(j+1)\pi}
		\ge \tfrac{\frac{1}{8}\pi}{2\pi^{-1}\katgl+1},
		\]
		where we used $j\pi \le \katgl$ and hence $\katgl+j\pi\le \katgl+(j+1)\pi \le 2\katgl+\pi$.
		
		If $d_0\not \in [0,\frac{1}{8}\pi^2]$, then $d_1=\frac{1}{2}\pi^2-d_0\in [0,\frac{3}{8}\pi^2]$ and hence, $\frac{1}{2}\pi^2 \pm d_1\in [\frac{1}{8}\pi^2,\frac{7}{8}\pi^2]$.
		By the definition of $d_1$, we have
		 $k^2=(m+\frac{1}{2})\pi^2\pm d_1$ for some $m\in \Z$. By $\tilde{\mu}_n=n\pi$ in \eqref{allcases} and $(\katgl)^2=|k^2-\tilde{\mu}_n^2|=|(m-n^2)\pi^2+(\frac{1}{2}\pi^2\pm d_1)|$, we have
		\[
		|\katgl-j\pi|=
		 \tfrac{|(\katgl)^2-j^2\pi^2|}
		{\katgl+j\pi}
		=\tfrac{|(N-j^2)\pi^2\pm (\frac{1}{2}\pi^2\pm d_1)|}{\katgl+j\pi}
		\ge \tfrac{\frac{1}{8}\pi^2}{\katgl+j\pi}
		\ge \tfrac{\frac{1}{8}\pi}{2\pi^{-1}\katgl+1},
		\]
		where $N:=m-n^2$ for $k^2\ge \tilde{\mu}_n^2$ or $N:=n^2-m$ for $k^2<\tilde{\mu}_n^2$,
		and
		\[
		 |\katgl-(j+1)\pi|=\tfrac{|(\katgl)^2-(j+1)^2\pi^2|}
		 {\katgl+(j+1)\pi}
		 =\tfrac{|((j+1)^2-N)\pi^2\pm  (\frac{1}{2}\pi^2 \pm d_1)|}{\katgl+(j+1)\pi}
		\ge \tfrac{\frac{1}{8}\pi}{2\pi^{-1}\katgl+1}.
		\]
		This proves \eqref{gl:lower} for the case $\bc_1=\bc_3$.
		
		We now consider $\bc_1\ne \bc_3$.
		Consider the unique integer $j$ such that $j \le \frac{\katgl}{\pi}+\frac{1}{2}<j+1$. Then it is obvious that $(j-\frac{1}{2})\pi \le \katgl<(j+\frac{1}{2})\pi$ and $\inf_{m\in \Z} |\katgl-(m+\frac{1}{2})\pi|=\min(|\katgl-(j-\frac{1}{2})\pi|,
|\katgl-(j+\frac{1}{2})\pi|)$.
		If $d_0\in [0, \frac{1}{8}\pi^2]\cup [\frac{3}{8}\pi^2,\frac{1}{2}\pi^2]$, then
		$-\frac{1}{4}\pi^2 \pm d_0\in
		 [-\frac{6}{8}\pi^2,-\frac{5}{8}\pi^2]\cup [-\frac{3}{8}\pi^2,-\frac{1}{8}\pi^2]\cup[\frac{1}{8}\pi^2, \frac{2}{8}\pi^2]$.
		By the definition of $d_0$, we have $k^2=m\pi^2 \pm d_0$ for some $m\in \Z$. Therefore, by $\tilde{\mu}_n=n\pi$ in \eqref{allcases} and $(\katgl)^2=|k^2-\tilde{\mu}_n^2|=|(m-n^2)\pi^2 \pm d_0|$, we have
		\[
		 |\katgl-(j-\tfrac{1}{2})\pi|
		 =\tfrac{|(\katgl)^2-(j-\frac{1}{2})^2\pi^2|}
		 {\katgl+(j-\frac{1}{2})\pi}
		=\tfrac{|(N-j^2+j)\pi^2+ (-\frac{1}{4}\pi^2\pm d_0)|}{\katgl+(j-\frac{1}{2})\pi}
		\ge \tfrac{\frac{1}{8}\pi}{2\pi^{-1}\katgl+1},
		\]
		where $N:=m-n^2$ for $k^2\ge \tilde{\mu}_n^2$ or $N:=n^2-m$ for $k^2<\tilde{\mu}_n^2$,
		and
		\[
		 |\katgl-(j+\tfrac{1}{2})\pi|=
		 \tfrac{|(\katgl)^2-(j+\frac{1}{2})^2\pi^2|}
		 {\katgl+(j+\frac{1}{2})\pi}
		=\tfrac{|(j^2+j-N)\pi^2- (-\frac{1}{4}\pi^2\pm d_0)|}{\katgl+(j+\frac{1}{2})\pi}
		\ge \tfrac{\frac{1}{8}\pi}{2\pi^{-1}\katgl+1},
		\]
		where we used $(j-\frac{1}{2})\pi \le \katgl$ and hence $\katgl+(j-\frac{1}{2})\pi\le \katgl+(j+\frac{1}{2})\pi \le 2\katgl+\pi$.
		
		If $d_0\not \in [0, \frac{1}{8}\pi^2]\cup [\frac{3}{8}\pi^2,\frac{1}{2}\pi^2]$, then
		 $d_1=\frac{1}{2}\pi^2-d_0
		\in [\frac{1}{8}\pi^2,\frac{3}{8}\pi^2]$ and consequently, $\pm(\frac{1}{4}\pi^2\pm d_1)\in [-\frac{5}{8}\pi^2,-\frac{3}{8}\pi^2]\cup [-\frac{1}{8}\pi^2,\frac{1}{8}\pi^2]
\cup [\frac{3}{8}\pi^2, \frac{5}{8}\pi^2]$, from which we obtain $\frac{1}{4}\pi^2\pm (\frac{1}{4}\pi^2\pm d_1) \in [-\frac{3}{8}\pi^2,-\frac{1}{8}\pi^2]\cup [\frac{1}{8}\pi^2,\frac{3}{8}\pi^2]\cup [\frac{5}{8}\pi^2,\frac{7}{8}\pi^2]$.
By the definition of $d_1$, we have $k^2=(m+\frac{1}{2})\pi^2 \pm d_1$ for some $m\in \Z$. By $\tilde{\mu}_n=(n+\frac{1}{2})\pi$ in \eqref{allcases} and $(\katgl)^2=|k^2-\tilde{\mu}_n^2|=	 |(m-n^2-n)\pi^2+\frac{1}{4}\pi^2\pm d_1|$, we have
		\[ |\katgl-(j-\tfrac{1}{2})\pi| =\tfrac{|(\katgl)^2-(j-\frac{1}{2})^2\pi^2|} {\katgl+(j-\frac{1}{2})\pi}
		 =\tfrac{|(N-j^2+j)\pi^2-[\frac{1}{4}\pi^2\pm (\frac{1}{4}\pi^2\pm d_1)]|}{\katgl+(j+1/2)\pi}
		\ge \tfrac{\frac{1}{8}\pi}{2\pi^{-1}\katgl+1},
		\]
		where $N:=m-n^2-n$ for $k^2\ge \tilde{\mu}_n^2$ or $N:=n^2+n-m$ for $k^2<\tilde{\mu}_n^2$,
		and
		\[
		 |\katgl-(j+\tfrac{1}{2})\pi|=
		 \tfrac{|(\katgl)^2-(j+\frac{1}{2})^2\pi^2|}
		 {\katgl+(j+\frac{1}{2})\pi}
		 =\tfrac{|(j^2+j-N)\pi^2+[\frac{1}{4}\pi^2\pm (\frac{1}{4}\pi^2\pm d_1)]|}{\katgl+(j+\frac{1}{2})\pi}
		\ge \tfrac{\frac{1}{8}\pi}{2\pi^{-1}\katgl+1}.
		\]
	This proves \eqref{gl:lower:2} for the case $\bc_1\ne \bc_3$.
	\end{proof}
	
\begin{lemma}\label{normsaux}
Consider the problem \eqref{1dYtil}. Define
\[
\tgl:=\sqrt{\left|1-\tfrac{\tilde{\mu}_{n}^{2}}{k^2}\right|} \quad \text{and} \quad
\rtgl:=
\begin{cases}
	\tgl &\text{if }\; \tilde{\mu}_n^2 \le k^2,\\
	i \tgl &\text{if }\; \tilde{\mu}_n^2 > k^2,
\end{cases}
\quad
\forall n \in \NN.
\]
where $\tilde{\mu}_n$ for $n \in \NN$ is given in \eqref{allcases}.
\begin{itemize}
	\item[(1)] Suppose that $\mathcal{B}_{1} \tilde{Y}_n(0)=-\tilde{Y}'_n(0)=1$ and
	$\mathcal{B}_3 \tilde{Y}_n(1) = \alpha\tilde{Y}'_n(1)+(1-\alpha) \tilde{Y}_n(1)=0$ with $\alpha=\{0,1\}$. Then, the solutions $\{\tilde{Y}_n\}_{n \in \NN}$ to the problem \eqref{1dYtil} satisfy
	\[
	\tilde{Y}_{n}(y) =
	\begin{cases}
		 \tfrac{-\alpha\cos(k\rtgl(y-1))-(1-\alpha)\sin(k\rtgl(y-1))}{k\rtgl(\alpha\sin(k\rtgl)+(1-\alpha)\cos(k\rtgl))}, & \text{if} \quad \tilde{\mu}_n^2 \neq k^2,\\
		\alpha(\tfrac{1}{2}y^2-y) -(1-\alpha)(y - 1), & \text{if} \quad \tilde{\mu}_n^2 = k^2.
	\end{cases}
	\]
	If $\mathcal{B}_1 \tilde{Y}_n(0) = -\alpha\tilde{Y}'_n(0)+(1-\alpha) \tilde{Y}_n(0)=0$ and $\mathcal{B}_{3} \tilde{Y}_n(1)=\tilde{Y}'_n(1)=1$ with $\alpha=\{0,1\}$, then the solutions to \eqref{1dYtil} are given above with $y$ replaced by $1-y$. Moreover, for both cases,
	the norms $\|\tilde{Y}_{n}\|_{0,\mathcal{I}}$ and $\|\tilde{Y}'_{n}\|_{0,\mathcal{I}}$ are given by
	\begin{align*}
	& \|\tilde{Y}_{n}\|^{2}_{0,\mathcal{I}}=
	\begin{cases}
		\tfrac{k\rtgl + (-1)^{1-\alpha} \sin(k\rtgl) \cos(k\rtgl)}{2 (k\rtgl)^{3}((1-\alpha)\cos^{2}(k\rtgl) + \alpha \sin^{2}(k\rtgl))}, &
		\text{if} \quad \tilde{\mu}_n^2 \neq k^2,\\
		\tfrac{1}{3} (1-\alpha) + \tfrac{2}{15} \alpha, & \text{if} \quad \tilde{\mu}_n^2 = k^2,
	\end{cases}\\
	& \|\tilde{Y}'_{n}\|^{2}_{0,\mathcal{I}}=
	\begin{cases}
		\tfrac{k\rtgl + (-1)^{\alpha} \sin(k\rtgl) \cos(k\rtgl)}{2 k\rtgl((1-\alpha)\cos^{2}(k\rtgl) + \alpha \sin^{2}(k\rtgl))}, &
		\text{if} \quad \tilde{\mu}_n^2 \neq k^2,\\
		(1-\alpha) + \tfrac{1}{3}\alpha, & \text{if} \quad \tilde{\mu}_n^2 = k^2.
	\end{cases}
	\end{align*}
	\item[(2)]
	Suppose that $\mathcal{B}_{1} \tilde{Y}_n(0)=\tilde{Y}_n(0)=1$ and
	$\mathcal{B}_3 \tilde{Y}_n(1) = \alpha\tilde{Y}'_n(1)+(1-\alpha) \tilde{Y}_n(1)=0$ with $\alpha=\{0,1\}$. Then, the solutions $\{\tilde{Y}_n\}_{n \in \NN}$ to the problem \eqref{1dYtil} satisfy
	\[
		\tilde{Y}_{n}(y) =
		 \tfrac{\alpha\cos(k\rtgl(y-1))+(1-\alpha)\sin(k\rtgl(1-y))}{\alpha\cos(k\rtgl)+(1-\alpha)\sin(k\rtgl)}.
	\]
	If $\mathcal{B}_1 \tilde{Y}_n(0) = -\alpha\tilde{Y}'_n(0)+(1-\alpha) \tilde{Y}_n(0)=0$ and $\mathcal{B}_{3} \tilde{Y}_n(1)=\tilde{Y}_n(1)=1$  with $\alpha=\{0,1\}$, then the solutions to \eqref{1dYtil} are given above with $y$ replaced by $1-y$. Moreover, for both cases,
	the norms $\|\tilde{Y}_{n}\|_{0,\mathcal{I}}$ and $\|\tilde{Y}'_{n}\|_{0,\mathcal{I}}$ are given by
	\[
		 \|\tilde{Y}_{n}\|^{2}_{0,\mathcal{I}} =						 \tfrac{k\rtgl+(-1)^{1-\alpha}\sin(k\rtgl)\cos(k\rtgl)}{2k\rtgl((1-\alpha)\sin^{2}(k\rtgl) + \alpha \cos^{2}(k\rtgl))}, \quad
		 \|\tilde{Y}'_{n}\|^{2}_{0,\mathcal{I}} =
		 \tfrac{k\rtgl((-1)^{\alpha}\sin(k\rtgl)\cos(k\rtgl)+k\rtgl)}{2((1-\alpha)\sin^{2}(k\rtgl)+\alpha\cos^{2}(k\rtgl))}.
	\]
	\end{itemize}
	\end{lemma}
	\begin{proof}
		The above solutions and norms can be obtained from direct calculations (similar to the proof of \cref{1dX:sol}).
	\end{proof}
		
\begin{proof}[Proof of \cref{traceL2}]
	Given that \eqref{allcases} holds, we know by \cref{lem:tYn} that each solution stated in \cref{normsaux} is well defined and hence each term of $\tilde{u}$ in \eqref{auxsol} is well defined. It is also straightforward to see that \eqref{auxsol} satisfies \eqref{helmholtz:aux}.
	
	By \eqref{Zn}, we recall that $\{\tilde{X}_{n}\}_{n \in \N_{0}}$, $\{\tilde{X}'_{n}\}_{n \in \N_{0}}$, $\{\tilde{X}''_{n}\}_{n \in \N_{0}}$ are orthogonal systems in $L^{2}(\mathcal{I})$. Also,  $\|\tilde{X}_n'\|_{0,\mathcal{I}}=\tilde{\mu}_n$ and $\|\tilde{X}_n''\|_{0,\mathcal{I}}=\tilde{\mu}_n^2$.
	Let $S_{M}(\tilde{u})$ denote a partial sum of $\tilde{u}$ with the $M$th term as its last term. Let $\au_x$ be the first partial derivative of $\au$ in the $x$ direction obtained by term-by-term differentiation. Define $\au_y$, $\au_{xx}$, and $\au_{xy}$ similarly.
 	
 	Suppose that $\mathcal{B}_{1} = \id$ and $\mathcal{B}_3 \in \{\id,\tfrac{\partial}{\partial \nu}\}$. We can pick $M \in \N$ such that for all $n \ge M$, we have $\tilde{\mu}_{n}^2 >k^2$, $\coth(k\tgl) \le 2$, and $\frac{1}{4} \le \left|1 - \tfrac{k^2}{\tilde{\mu}_n^2} \right|$ such that item (2) of \cref{normsaux} with $\alpha=0$ (i.e., $\mathcal{B}_3=\id$) implies
 	\[
 	\begin{aligned}
 	& \|\tilde{Y}_n\|_{0,\mathcal{I}}^2 = \tfrac{\sinh(k\tgl)\cosh(k\tgl)-k\tgl}{2 k\tgl \sinh^2(k\tgl)} \le \tfrac{\coth(k\tgl)}{2k\tgl} \le (k\tgl)^{-1} =  \tilde{\mu}_n^{-1} \left|1 - \tfrac{k^2}{\tilde{\mu}_n^2} \right|^{-1/2} \le 2 \tilde{\mu}_n^{-1},\\
 	& \|\tilde{Y}'_n\|_{0,\mathcal{I}}^2 = \tfrac{k\tgl(\sinh(k\tgl)\cosh(k\tgl)+k\tgl)}{2\sinh^2(k\tgl)} \le \tfrac{3}{2} k\tgl = \tfrac{3}{2}\tilde{\mu}_n \left|1 - \tfrac{k^2}{\tilde{\mu}_n^2} \right|^{1/2} \le \tfrac{3}{2}\tilde{\mu}_n,
 	\end{aligned}
 	\]
	and item (2) of \cref{normsaux} with $\alpha=1$ (i.e., $\mathcal{B}_3=\tfrac{\partial}{\partial \nu}$) implies
	\[
	\begin{aligned}
	& \|\tilde{Y}_n\|_{0,\mathcal{I}}^2 =  \tfrac{\sinh(k\tgl)\cosh(k\tgl) + k\tgl}{2 k\tgl \cosh^2(k\tgl)} \le \tfrac{\tanh(k\tgl)}{k\tgl} \le \tilde{\mu}_n^{-1} \left|1 - \tfrac{k^2}{\tilde{\mu}_n^2} \right|^{-1/2} \le 2 \tilde{\mu}_n^{-1}, \\
	& \|\tilde{Y}'_n\|_{0,\mathcal{I}}^2 = \tfrac{k\tgl(\sinh(k\tgl)\cosh(k\tgl)-k\tgl)}{2\cosh^2(k\tgl)} \le \tfrac{1}{2} k\tgl = \tfrac{1}{2}\tilde{\mu}_n \left|1 - \tfrac{k^2}{\tilde{\mu}_n^2} \right|^{1/2} \le \tfrac{1}{2}\tilde{\mu}_n.
	\end{aligned}
	\]
 	That is given $\mathcal{B}_{1} = \id$ and $\mathcal{B}_3 \in \{\id,\tfrac{\partial}{\partial \nu}\}$, $\|\tilde{Y}_{n}\|_{0,\mathcal{I}}^{2} \le (k\tgl)^{-1} \le 2 \tilde{\mu}_{n}^{-1}$ and $\|\tilde{Y}'_{n}\|_{0,\mathcal{I}}^{2} \le k\tgl \le \tfrac{3}{2} \tilde{\mu}_{n}$ for all $n \ge M$.
 	Furthermore,
	\begin{align*}
		& \|\au_{x} - (S_{M}(\au))_x\|_{0,\Omega}
		= \sum_{n=M}^{\infty} |\wh{g}_{1}(n) \tilde{\mu}_{n}|^{2} \|\tilde{Y}_{n}\|^{2}_{0,\mathcal{I}},
		& \|\au_{xx} - (S_{M}(\au))_{xx}\|_{0,\Omega}
		= \sum_{n=M}^{\infty}|\wh{g}_{1}(n) \tilde{\mu}^2_{n}|^{2} \|\tilde{Y}_{n}\|^{2}_{0,\mathcal{I}},\\
		& \|\au_{y} - (S_{M}(\au))_{y}\|_{0,\Omega}
		= \sum_{n=M}^{\infty} |\wh{g}_{1}(n)|^{2} \|\tilde{Y}'_{n}\|^{2}_{0,\mathcal{I}},
		& \|\au_{xy} - (S_{M}(\au))_{xy}\|_{0,\Omega}
		= \sum_{n=M}^{\infty} |\wh{g}_{1}(n) \tilde{\mu}_{n}|^{2} \|\tilde{Y}'_{n}\|^{2}_{0,\mathcal{I}},
	\end{align*}
	Since $g_{1} \in \mathcal{Z}^{\frac{3}{2}}(\Gamma_1)$ for $\mathcal{B}_{1} = \id$, the above inequalities all tend to zero as $M \rightarrow \infty$.
	
	Now suppose that $\mathcal{B}_{1} = \frac{\partial}{\partial \nu}$ and $\mathcal{B}_3 \in \{\id,\tfrac{\partial}{\partial \nu}\}$. We can pick $M \in \NN$ such that for all $n \ge M$, we have $\tilde{\mu}_n^2 > k^2$, $\sin(2k\tgl) \le 2(\cosh(2k\tgl)-1)$, and $2^{-2/3} \le \left|1 - \tfrac{k^2}{\tilde{\mu}_n^2} \right|$ such that item (1) of \cref{normsaux} with $\alpha=0$ (i.e., $\mathcal{B}_3=\id$) implies
	\[
	\begin{aligned}
		& \|\tilde{Y}_n\|^{2}_{0,\mathcal{I}} = \tfrac{\sinh(k\tgl) \cosh(k\tgl)  - k\tgl}{2(k\tgl)^3 \cosh^2(k\tgl)} \le \tfrac{1}{2} (k\tgl)^{-3} \le \tfrac{1}{2}\tilde{\mu}_{n}^{-3} \left|1 - \tfrac{k^2}{\tilde{\mu}_n^2} \right|^{-3/2} \le \tilde{\mu}_{n}^{-3},\\
		& \|\tilde{Y}'_n\|^{2}_{0,\mathcal{I}} = \tfrac{\sinh(k\tgl)\cosh(k\tgl)+k\tgl}{2k\tgl\cosh^2(k\tgl)} \le \tilde{\mu}_{n}^{-1} \left|1 - \tfrac{k^2}{\tilde{\mu}_n^2} \right|^{-1/2} \le 2^{1/3}\tilde{\mu}_{n}^{-1},
	\end{aligned}
	\]
	and item (1) of \cref{normsaux} with $\alpha=1$ (i.e., $\mathcal{B}_3=\tfrac{\partial}{\partial \nu}$) implies
	\[
	\begin{aligned}
		& \|\tilde{Y}_n\|^{2}_{0,\mathcal{I}} = \tfrac{\sinh(k\tgl)\cosh(k\tgl) + k\tgl}{2(k\tgl)^3 \sinh^2(k\tgl)} \le \tfrac{3}{2} (k\tgl)^{-3} \le \tfrac{3}{2} \tilde{\mu}_{n}^{-3} \left|1 - \tfrac{k^2}{\tilde{\mu}_n^2} \right|^{-3/2} \le 3\tilde{\mu}_{n}^{-3},\\
		& \|\tilde{Y}'_n\|^{2}_{0,\mathcal{I}} = \tfrac{\sinh(k\tgl)\cosh(k\tgl)-k\tgl}{2k\tgl\sinh^2(k\tgl)} \le \tilde{\mu}_{n}^{-1} \left|1 - \tfrac{k^2}{\tilde{\mu}_n^2} \right|^{-1/2} \le 2^{1/3}\tilde{\mu}_{n}^{-1}.
	\end{aligned}
	\]
	That is given $\mathcal{B}_{1} = \frac{\partial}{\partial \nu}$ and $\mathcal{B}_3 \in \{\id,\tfrac{\partial}{\partial \nu}\}$, $\|\tilde{Y}_{n}\|_{0,\mathcal{I}}^{2} \le (k\tgl)^{-3} \le 3 \tilde{\mu}_{n}^{-3}$ and $\|\tilde{Y}'_{n}\|_{0,\mathcal{I}}^{2} \le (k\tgl)^{-1} \le 2^{1/3} \tilde{\mu}_{n}^{-1}$ for all $n \ge M$. Since $g_{1} \in \mathcal{Z}^{\frac{1}{2}}(\Gamma_1)$ for $\mathcal{B}_{1} = \frac{\partial}{\partial \nu}$, the above inequalities all tend to zero as $M \rightarrow \infty$. Therefore, in both of the previously discussed cases, we have shown that $\au_{x} \in L^{2}(\Omega)$ and $(S_{M}(\au))_x$ converges to $\au_{x}$ in $L^{2}(\Omega)$; the same implications also hold for the other three cases involving $\au_{xx}$, $\au_{y}$, and $\au_{xy}$.
	
	Clearly, $\au \in H^{1}(\Omega)$. By the trace inequality \cite[Theorem 3.37]{M00}, we have $\|\au\|_{1/2,\partial \Omega} \le C \|\au\|_{1,\Omega} < \infty$ for some constant $C$. Also, by the multiplicative trace inequality \cite[Theorem 1.5.10 and the last inequality of p. 41]{G85}, we have $\|\au_{x}\|_{0,\partial \Omega}^{2} \le C \|\au_{x}\|_{1,\Omega}\|\au_{x}\|_{0,\Omega} < \infty$ for some other constant $C$.
\end{proof}	
	
\begin{proof}[Proof of \cref{thm:stability3}]
	Let  $N_p := \max\{n \in \mathbb{N}_0:\tilde{\mu}_n^2 < k^2\}$, $N_c \in \mathbb{N}$ be such that $\tilde{\mu}_{N_c}^2 = k^2$, and $N_e := \min\{n \in \mathbb{N}:\tilde{\mu}_n^2 > k^2\}$. Recall that  $\tgl:=\sqrt{\left|1-\frac{\tilde{\mu}_{n}^{2}}{k^2}\right|}$ and observe that $\|\tilde{X}_n'\|_{0,\mathcal{I}}=\tilde{\mu}_n$. By \eqref{Zn}, since $\{\tilde{X}_{n}\}_{n \in \N_{0}}$ and $\{\tilde{X}'_{n}\}_{n \in \N_{0}}$ are orthogonal systems in $L^{2}(\mathcal{I})$, we deduce from \eqref{auxsol} that
	\begin{align} \nonumber
		\|\nabla \au \|^{2}_{0,\Omega} & + k^{2} \|\au\|^{2}_{0,\Omega} = \left\|\sum_{n=0}^{\infty} \wh{g_{1}}(n) \tilde{X}'_{n} \tilde{Y}_{n}\right\|^{2}_{0,\Omega} + \left\|\sum_{n=0}^{\infty} \wh{g_{1}}(n) \tilde{X}_{n}\tilde{Y}'_{n}\right\|^{2}_{0,\Omega} + k^{2} \left\|\sum_{n=0}^{\infty} \wh{g_{1}}(n) \tilde{X}_{n}\tilde{Y}_{n}\right\|^{2}_{0,\Omega}\\
		& \le \sum_{n=0}^{\infty} |\wh{g_{1}}(n) \tilde{\mu}_{n}|^{2} \|\tilde{Y}_{n}\|^{2}_{0,\mathcal{I}} + \sum_{n=0}^{\infty} |\wh{g_{1}}(n)|^{2} \|\tilde{Y}'_{n}\|^{2}_{0,\mathcal{I}} + k^{2} \sum_{n=0}^{\infty} |\wh{g_{1}}(n)|^{2} \|\tilde{Y}_{n}\|^{2}_{0,\mathcal{I}}.
	 \label{normXY:tilde}
	\end{align}
	Regrouping the terms, we have
	\begin{equation} \label{ubginaux}
		\sum_{n=0}^{\infty} |\wh{g_{1}}(n)|^{2} (\|\tilde{Y}'_{n}\|^{2}_{0,\mathcal{I}} + (\tilde{\mu}_{n}^{2} + k^{2}) \|\tilde{Y}_{n}\|^{2}_{0,\mathcal{I}})
		\le
		\max\left\{\max_{0 \le n \le N_p} \tilde{\phi}_{n},\tilde{\theta}_{N_c},\max_{n \ge N_e} \tilde{\psi}_{n}\right\}\sum_{n=0}^{\infty} |\wh{g_{1}}(n)|^{2},
	\end{equation}
	where $\tilde{\phi}_{n}$, $\tilde{\theta}_{N_c}$, and $\tilde{\psi}_{n}$ are defined similar to \eqref{phinpsin} with $X_{n}$, $x$, and $\mu_n$ being replaced by $\tilde{Y}_n$, $y$, and $\tilde{\mu}_n$ respectively.
	
	Item (1): Suppose $\mathcal{B}_{1} = \frac{\partial}{\partial \nu}$  and $\mathcal{B}_{3} = \frac{\partial}{\partial \nu}$. Define $z_n:=k\tgl$ for $n \in \NN$. From \eqref{gl:lower}, we observe that if the infimum on the left-hand side of the inequality occurs at $j=0$, then
	\be \label{lbkatgl}
	\tfrac{\pi}{2} > z_n \ge \tfrac{\pi}{8 (1+2\pi^{-1} z_n)}  \ge \tfrac{\pi}{16}, \quad \text{that is,} \quad z_n \ge \tfrac{\pi}{16} \quad \forall n \in \NN,k>0.
	\ee
	Otherwise, if the infimum on the left-hand side of \eqref{gl:lower} occurs at a nonzero $j$, then $z_n \ge \frac{\pi}{2}$. Using item (i) of \cref{normsaux} with $\alpha=1$, we obtain
	\[
			\tilde{\phi}_n=\tfrac{k^2 + (k^2-z_n^2) \tfrac{\sin(2z_n)}{2z_n}}{z_n^2 \sin^{2}(z_n)}, \quad
 		 \tilde{\psi}_n=\tfrac{2((k^2+z_n^2)\frac{\sinh(2z_n)}{2z_n}+ k^2)}{z_n^2\left(\cosh(2z_n)-1\right)}.
	\]
	For each $n \in \NN$, let $\gamma_n=\arginf_{j \in \Z} |z_n - j\pi|$. By \eqref{gl:lower}, we have
	\be \label{lbsin}
	\sin^{2}(z_n) = \sin^2(z_n - \gamma_n \pi) \ge \tfrac{4}{\pi^2} (z_n-\gamma_n\pi)^2  \ge \tfrac{1}{16(1+2\pi^{-1}z_n)^2}, \quad \forall z_n \in [(\gamma_n - \tfrac{1}{2}) \pi,(\gamma_n+\tfrac{1}{2})\pi].
	\ee 
	For $k>0$, $n \le N_p$, and $z_n \in (0,k]$, we have
	\begin{align*}
	\tilde{\phi}_n & \le \tfrac{2k^2 - z_n^2}{z_n^2 \sin^{2}(z_n)} \le 32 \tfrac{k^2}{z_n^2} (1+2\pi^{-1} z_n)^{2} \le 32 k^2\left(\tfrac{1}{z_n^2} + \tfrac{4}{z_n\pi} + \tfrac{4}{\pi^2}\right) \le 32 \left(\tfrac{256}{\pi^2} + \tfrac{64}{\pi^2} + \tfrac{4}{\pi^2}\right)\max\{k^2,1\}\\
	& \le \tfrac{10368}{\pi^2} \max\{k^2,1\} \le 1051 \max\{k^2,1\},
	\end{align*}
	where we used \eqref{lbsin} to arrive at the second inequality, and applied \eqref{lbkatgl} to arrive at the fourth inequality. Next for $k>0$, $n \ge N_e$, and $z_n\in (0,\infty)$, we have
	\begin{align*}
	\tilde{\psi}_n &
	\le 2\left(\tfrac{k^2}{z_n^2}+1\right) \left(\tfrac{\cosh(2z_n)+1}{\cosh(2z_n)-1}\right)
	\le 2  \left(\tfrac{256}{\pi^2}+1\right)\left(\tfrac{\cosh(\tfrac{\pi}{8})+1}{\cosh(\tfrac{\pi}{8})-1}\right) \max\{k^2,1\} \le 1434 \max\{k^2,1\},
	\end{align*}
	where we used \eqref{lbkatgl} to arrive at the second inequality. Consequently,
	\[
	\max\left\{\max_{0 \le n \le N_p} \tilde{\phi}_{n},\tilde{\theta}_{N_c},\max_{n \ge N_e} \tilde{\psi}_{n}\right\} \le  \max\left\{1051 \max\{k^2,1\},\tfrac{4}{15}k^2+\tfrac{1}{3},1434 \max\{k^2,1\}\right\} = 1434 \max\{k^2,1\}.
	\]
	Applying the Parseval's identity to \eqref{ubginaux} and finally using \eqref{ineq:ab}, we have \eqref{Ck:auxsol:neum}.
	
	Item (1): suppose $\mathcal{B}_{1} = \frac{\partial}{\partial \nu}$ and $\mathcal{B}_{3} = \id$. Using item (i) of \cref{normsaux} with $\alpha=0$, we obtain
 \[
 		\tilde{\phi}_n=\tfrac{k^2 -(k^2-z_n^2) \frac{\sin(2z_n)}{2z_n}}{
	 	z_n^2 \cos^{2}(z_n)},
 		\quad
		\tilde{\psi}_n= \tfrac{k^{2} (\sinh(2z_n) -2z_n) + z_n^2 \sinh(2z_n)}{z_n^3(\cosh(2z_n)+1)}.
	\]
	To obtain an upper bound for $\tilde{\phi}_n$, we shall list several observations, which are used in its estimation. For each $n \in \NN$, let $\gamma_n=\arginf_{j \in \Z} |z_n -(j+\tfrac{1}{2})\pi|$. By \eqref{gl:lower:2}, we have
	\be \label{lbcos}
	\cos^2(z_n) = \sin^2(z_n-(\gamma_n+\tfrac{1}{2})\pi) \ge \tfrac{4}{\pi^2} (z_n-(\gamma_n+\tfrac{1}{2})\pi)^2  \ge \tfrac{1}{16(1+2\pi^{-1}z_n)^2}, \quad \forall z_n \in [\gamma_n \pi,(\gamma_n+1)\pi].
	\ee
	Also,
	\begin{align*}
		 \tfrac{d}{dz_n}\left(\tfrac{\left(1-\frac{\sin(2z_n)}{2z_n}\right)}{z_n^2\cos^2(z_n)}\right) & = \tfrac{3\cos(z_n)z_n\left(\frac{\sin(2z_n)}{2z_n}+ \frac{2}{3} z_n \tan(z_n) -1 \right)}{z_n^4 \cos^{3}(z_n)}
		\ge \tfrac{3\left(1-\frac{2}{3}z_n^2 + \frac{2}{3}(z_n^2 + \frac{1}{3}z_n^4)-1\right)}{z_n^3 \cos^{2}(z_n)} = \tfrac{2z_n}{3 \cos^{2}(z_n)} \ge 0, \quad \forall z_n \in (0,1],
	\end{align*}
	where we used \eqref{sincos} and $z_n^2 + \tfrac{1}{3}z_n^4 \le z_n \tan(z_n)$ for all $z_n \in (0,1]$. Now, for $k>0$, $n \le N_p$, and $z_n \in (0,k]$, we have
	\begin{align*}
	\tilde{\phi}_n & \le \left(\tfrac{1-\frac{\sin(2z_n)}{2z_n}}{z_n^2\cos^2(z_n)}k^2 + \tfrac{\sin(2z_n)}{2 z_n \cos^{2}(z_n)}\right) \bchi_{\{z_n \le 1\}} + \tfrac{2}{z_n^2\cos^2(z_n)}k^2 \bchi_{\{1<z_n \le k\}} \\
	& \le \left(\tfrac{1-\frac{\sin(2)}{2}}{\cos^2(1)}k^2 + \tfrac{1}{\cos^{2}(1)}\right) \bchi_{\{z_n \le 1\}} + \tfrac{32(1+2\pi^{-1} z_n)^2}{z_n^2}k^{2} \bchi_{\{1<z_n \le k\}}\\
	& \le \tfrac{4-\sin(2)}{2\cos^2(1)} \max\{k^2,1\} \bchi_{\{z_n \le 1\}} + 32 (1+2\pi^{-1})^2 \max\{k^2,1\}\bchi_{\{1<z_n \le k\}}\le 86 \max\{k^2,1\},
	\end{align*}
	where we used \eqref{lbcos} to arrive at the second term of the first inequality. Next, note that for all $z_n>0$
	{\small
	\begin{align*}
		& \tfrac{d}{dz_n} \left(\tfrac{ (\sinh(2z_n) -2z_n)}{2z_n^3(\cosh(2z_n)+1)}\right) = \tfrac{6 \sinh(z_n) \cosh^{3}(z_n)}{z_n^{4} (\cosh(2z_n)+1)^2} \left(-1 + \tfrac{2z_n^2}{3\cosh^2(z_n)} + \tfrac{2z_n}{\sinh(2z_n)}\right) \\
		& \quad \le \tfrac{6 \sinh(z_n) \cosh^{3}(z_n)}{z_n^{4} (\cosh(2z_n)+1)^2} \left(
		(-1 + \tfrac{2}{3} (z_n^2 - z_n^4 + \tfrac{2}{3} z_n^6) + \left(1 - \tfrac{2}{3} z_n^2 + \tfrac{14}{45} z_n^4\right))\bchi_{\{z_n \le \frac{2}{\sqrt{5}}\}} + \left(-\tfrac{2}{3} + \tfrac{4}{\sqrt{5}\sinh(\tfrac{4}{\sqrt{5}})}\right)\bchi_{\{z_n>\frac{2}{\sqrt{5}}\}}\right)\\
		& \quad \le 0,
	\end{align*}
	}
	where we used $\tfrac{z_n^2}{\cosh^2(z_n)} \le z_n^2 - z_n^4 + \tfrac{2}{3} z_n^6$ and $\tfrac{2z_n}{\sinh(2z_n)} \le 1 - \tfrac{2}{3} z_n^2 + \tfrac{14}{45} z_n^4$ for all $z_n \in (0,\tfrac{2}{\sqrt{5}}]$.
	Now, for $k>0$, $n \ge N_e$, and $z_n \in (0,\infty)$, we have
	\[
	\tilde{\psi}_n \le \lim_{z_n \rightarrow 0} \left(\tfrac{ \sinh(2z_n) -2z_n}{z_n^3(\cosh(2z_n)+1)}k^{2} + \tfrac{\sinh(2z_n)}{z_n(\cosh(2z_n)+1)}\right)
	= \tfrac{2}{3}k^2 + 1 \le 2 \max\{k^2,1\}.
	\]
	Consequently,
	\[
	\max\Big\{\max_{0 \le n \le N_p} \tilde{\phi}_{n},\tilde{\theta}_{N_c},\max_{n \ge N_e} \tilde{\psi}_{n}\Big\} \le
	\max\left\{86 \max\{k^2,1\},\tfrac{2}{3}k^{2}+1,2\max\{k^2,1\}\right\} = 86  \max\{k^2,1\}.
	\]
	Applying the Parseval's identity to \eqref{ubginaux}, and finally using \eqref{ineq:ab}, we have \eqref{Ck:auxsol:neum}.
	
	Item (2): suppose $\mathcal{B}_{1} = \id$ and $\mathcal{B}_{3} = \frac{\partial}{\partial \nu}$. Continuing from \eqref{normXY:tilde}, we have
	\begin{align}
		\nonumber
		& \sum_{n=0}^{\infty} |\wh{g_{1}}(n) \tilde{\mu}_{n}|^{2} \|\tilde{Y}_{n}\|^{2}_{0,\mathcal{I}} + \sum_{n=0}^{\infty} |\wh{g_{1}}(n)|^{2} \|\tilde{Y}'_{n}\|^{2}_{0,\mathcal{I}} + k^{2} \sum_{n=0}^{\infty} |\wh{g_{1}}(n)|^{2} \|\tilde{Y}_{n}\|^{2}_{0,\mathcal{I}} = \sum_{n=N_e}^{\infty} k^{2} |\wh{g_{1}}(n)|^{2}  \|\tilde{Y}_{n}\|^{2}_{0,\mathcal{I}}\\
		\nonumber
		& \qquad + \sum_{n=0}^{N_e-1} |\wh{g_{1}}(n)|^{2} \left( (\tilde{\mu}_{n}^{2}+k^{2}) \|\tilde{Y}_{n}\|^{2}_{0,\mathcal{I}} + \|\tilde{Y}'_{n}\|^{2}_{0,\mathcal{I}} \right) + \sum_{n=N_e}^{\infty} |\wh{g_{1}}(n)|^{2} \tilde{\mu}_{n} \left(\tfrac{ \tilde{\mu}_{n}^{2} \|\tilde{Y}_{n}\|^{2}_{0,\mathcal{I}}  + \|\tilde{Y}'_{n}\|^{2}_{0,\mathcal{I}} }{\tilde{\mu}_{n}}\right)\\
		&
		\label{ubginauxd}
		\quad \le \max\left\{\max_{0 \le n \le N_p} \tilde{\phi}_{n},\tilde{\theta}_{N_c}, k^{2}\max_{n\ge N_e} \|\tilde{Y}_{n}\|^{2}_{0,\mathcal{I}}\right\}\sum_{n=0}^{\infty} |\wh{g_{1}}(n)|^{2} + \max_{n\ge N_e} \left(\tfrac{\tilde{\mu}_{n}^{2} \|\tilde{Y}_{n}\|^{2}_{0,\mathcal{I}}  + \|\tilde{Y}'_{n}\|^{2}_{0,\mathcal{I}}}{\tilde{\mu}_{n}}\right) \sum_{n=0}^{\infty} |\wh{g_{1}}(n)|^{2} \tilde{\mu}_{n}.
	\end{align}
	Now, item (ii) of \cref{normsaux} with $\alpha=1$, we have for $k>0$, $0 \le n \le N_p$, and $z_n \in (0,k]$
	\begin{align*}
		\tilde{\phi}_{n} & = \tfrac{k^{2} +(k^{2}-z_n^2)\frac{\sin(2z_n)}{2z_n}}{ \cos^{2}(z_n)} \le \tfrac{2}{\cos^2(1)} \max\{k^2,1\} \bchi_{\{z_n \le 1\}} + 32k^2(1+2\pi^{-1}z_n)^2 \bchi_{\{1<z_n\le k\}}\\
		&\le \tfrac{2}{\cos^2(1)} \max\{k^2,1\} \bchi_{\{z_n \le 1\}} + 32k^2(1+2\pi^{-1}k)^2 \bchi_{\{1<z_n\le k\}}\\
		& \le \tfrac{2}{\cos^2(1)} \max\{k^2,1\} \bchi_{\{z_n \le 1\}} + 32(1+2\pi^{-1})^2  \max\{k^4,1\} \bchi_{\{1<z_n\le k\}}\\
		& \le \max\left\{ \tfrac{2}{\cos^2(1)} ,  32(1+2\pi^{-1})^2 \right\} \max\{k^4,1\} \le 86\max\{k^4,1\},
	\end{align*}
	where we used \eqref{lbcos} to arrive at the second term of the first inequality.
	Next, for $k>0$, $n\ge N_e$, and $z_n \in (0,\infty)$, we have
$ k^{2}\|\tilde{Y}_{n}\|^{2}_{0,\mathcal{I}} =  \tfrac{\frac{\sinh(2z_n)}{2z_n}+1}{\cosh(2z_n)+1} k^2\le k^2.
$
	Consequently,
	\begin{equation} \label{thm3:3max:item1}
	\max\left\{\max_{0 \le n \le N_p} \tilde{\phi}_{n},\tilde{\theta}_{N_c},k^2\max_{n \ge N_e} \|\tilde{Y}_{n}\|^{2}_{0,\mathcal{I}}\right\} \le
	\max\left\{86 \max\{k^4,1\}, 2k^2, k^{2}\right\} = 86 \max\{k^4,1\}.
	\end{equation}
	Applying the Parseval's identity and using \eqref{ineq:ab}, we have the first term of the right-hand side of \eqref{Ck:auxsol:diric}. Next, we recall that since $\{\tilde{\mu}_n=n\pi\}_{n \in \NN}$ or $\{\tilde{\mu}_n=(n+\frac{1}{2})\pi\}_{n \in \NN}$, the lower bound in \eqref{lbgl} still holds with $\mu_n,\gl$ replaced by $\tilde{\mu}_n,\tgl$ respectively. We also note that for all $k >0$ and $z_n \in (0,\infty)$
	\begin{align*}
	\tfrac{(z_n^2 + k^2)^{\frac{1}{2}}\sinh(z_n)}{z_n \cosh(z_n)} & \le \left(1 + \tfrac{k^{2}}{z_n^2}\right)^{\frac{1}{2}} \tanh(z_n)\bchi_{\{k\ge 1,z_n>1\}\cup \{k< 1,z_n>0\}}
 	+ \left(\tfrac{z_n^2}{k^2} + 1\right)^{\frac{1}{2}}\left(\tfrac{\tanh(z_n)}{z_n}\right)k\bchi_{\{k\ge 1,z_n \le 1\}}\\
 	& \le \sqrt{2} k \bchi_{\{k\ge 1,z_n>1\}} + (1+\tfrac{1}{\eta^2})^{\frac{1}{2}} \bchi_{ \{k< 1,z_n>0\}} + \sqrt{2}k\bchi_{\{k\ge 1,z_n \le 1\}}\\
 	& \le \max\left\{\sqrt{2}, (1+\tfrac{1}{\eta^2})^{\frac{1}{2}} \right\}  \max\{k,1\} \le 2\max\{k,1\},
	\end{align*}
	where we used the lower bound in \eqref{lbgl} to arrive at the second term of the second inequality. Now, for $k>0$ and $n \ge N_e$, we have
	\begin{equation} \label{thm3:term2:item1}
 		\tfrac{\tilde{\mu}_{n}^{2} \|\tilde{Y}_{n}\|^{2}_{0,\mathcal{I}}  + \|\tilde{Y}'_{n}\|^{2}_{0,\mathcal{I}}}{\tilde{\mu}_{n}}
		 =\tfrac{(2z_n^{2} + k^{2}) \sinh(2z_n) + 2 k^{2} z_n}{2z_n
		 (z_n^{2} + k^{2})^{\frac{1}{2}}(\cosh(2z_n) +1)}
		\le \tfrac{(z_n^2 + k^2)^{\frac{1}{2}}\sinh(z_n)}{z_n \cosh(z_n)} + \tfrac{k^2}{(z_n^2 + k^2)^{\frac{1}{2}} (\cosh(2z_n)+1)} \le 3 \max\{k,1\}.
	\end{equation}
	Recall that we assumed $g_{1} \in \mathcal{Z}^{\frac{1}{2}}(\Gamma_1)$. Plugging \eqref{thm3:3max:item1}-\eqref{thm3:term2:item1} into \eqref{ubginauxd} and using \eqref{ineq:ab}, we have the second term of the right-hand side of \eqref{Ck:auxsol:diric}.
	
	Item (2): suppose $\mathcal{B}_{1} = \id$ and $\mathcal{B}_{3} = \id$.
	Now, by item (ii) of \cref{normsaux} with $\alpha=0$, we have for $k>0$, $0 \le n \le N_p$, and $z_n \in (0,k]$
	\begin{align*}
	\tilde{\phi}_n & = \tfrac{k^{2} - (k^{2}-z_n^{2})\frac{\sin(2z_n)}{2z_n}}{\sin^{2}(z_n)} \le \tfrac{2}{\sin^{2}(z_n)}k^2 \le 32 k^2 (1 + 2\pi^{-1} z_n)^2 \le 32k^2 (1 + 2\pi^{-1} k)^2 \\
	& \le  32 (1 + 2\pi^{-1})^2  \max\{k^4,1\} \le 86 \max\{k^4,1\},
	\end{align*}
	where we used \eqref{lbsin} to arrive at the second inequality. Next for $n \ge N_e$ and $z_n > 0$, we have
$ k^{2}\|\tilde{Y}_n\|_{0,\mathcal{I}}^2 \le \tfrac{(\sinh(2z_n)-2z_n)}
	{2z_n(\cosh(2z_n)-1)}k^{2} \le k^2.
$
	Consequently,
	\begin{equation} \label{thm3:3max:item2}
	\max\Big\{\max_{0 \le n \le N_p} \tilde{\phi}_{n},\tilde{\theta}_{N_c},\max_{n \ge N_e} \tilde{\psi}_{n}\Big\} \le
	\max\left\{86 \max\{k^4,1\},\tfrac{2}{3}k^2 + 1,k^2\right\} = 86 \max\{k^4,1\}.
	\end{equation}
	Applying the Parseval's identity and using \eqref{ineq:ab}, we have the first term of the right-hand side of \eqref{Ck:auxsol:diric}. As before, since $\{\tilde{\mu}_n=n\pi\}_{n \in \NN}$ or $\{\tilde{\mu}_n=(n+\frac{1}{2})\pi\}_{n \in \NN}$, the lower bound in \eqref{lbgl} still holds with $\mu_n,\gl$ replaced by $\tilde{\mu}_n,\tgl$ respectively. Now, for $k >0$ and $n \ge N_e$, we have
	\begin{align}
		\nonumber
		& \tfrac{\tilde{\mu}_{n}^{2} \|\tilde{Y}_{n}\|^{2}_{0,\mathcal{I}}  + \|\tilde{Y}'_{n}\|^{2}_{0,\mathcal{I}}}{\tilde{\mu}_{n}}
		=\tfrac{(2z_n^{2} + k^{2}) \sinh(2 z_n) - 2 k^{2} z_n}{2z_n
		 (z_n^{2} + k^{2})^{\frac{1}{2}}(\cosh(2z_n) - 1)}\\
 		\nonumber
 		& \quad \le \tfrac{\left(1 + \tfrac{k^{2}}{z_n^{2}}\right)^{1/2} \sinh(2 z_n)}{(\cosh(2z_n) - 1)} \bchi_{\{k\ge 1, z_n>1\}}
 		+ \tfrac{2k\left(1 + \tfrac{z_n^{2}}{k^{2}}\right)^{1/2} \sinh(2 z_n)}{2z_n(\cosh(2z_n) - 1)} \bchi_{\{k\ge 1, z_n\le1\}} + \tfrac{\left(1 + \tfrac{k^{2}}{z_n^{2}}\right)^{1/2} \sinh(2 z_n)}{(\cosh(2z_n) - 1)} \bchi_{\{k< 1\}}\\
 		& \nonumber
 		\quad \le \tfrac{\sqrt{2}\sinh(2)}{(\cosh(2) - 1)}k\bchi_{\{k\ge 1, z_n>1\}} + \tfrac{2\sqrt{2}\cosh(\frac{\pi}{8})}{(\cosh(\frac{\pi}{8}) - 1)} k \bchi_{\{k\ge 1, z_n \le 1\}}
		+ \tfrac{\left(1 + \eta^{-2}\right)^{1/2} \sinh(2 \eta)}{(\cosh(2\eta) - 1)} \bchi_{\{k< 1\}} \\
		\label{thm3:term2:item2}
		& \quad \le \max\Big\{ \tfrac{\sqrt{2}\sinh(2)}{(\cosh(2) - 1)} ,  \tfrac{2\sqrt{2}\cosh(\frac{\pi}{8})}{(\cosh(\frac{\pi}{8}) - 1)},  \tfrac{\left(1 + \eta^{-2}\right)^{1/2} \sinh(2 \eta)}{(\cosh(2\eta) - 1)}  \Big\} \max\{k,1\} \le 40 \max\{k,1\},
	\end{align}
	where we used \eqref{lbkatgl} and \eqref{lbgl} to arrive at the second and third terms of the second inequality. Recall that we assumed $g_{1} \in \mathcal{Z}^{\frac{1}{2}}(\Gamma_1)$. Plugging \eqref{thm3:3max:item2}-\eqref{thm3:term2:item2} into \eqref{ubginauxd} and using \eqref{ineq:ab}, we have the second term of the right-hand side of \eqref{Ck:auxsol:diric}.
\end{proof}

\end{document}